\documentclass[reqno]{amsart}

\usepackage{amsmath,amssymb,graphicx}
\usepackage{verbatim} 

\usepackage[latin1]{inputenc}
\usepackage{enumerate}
\usepackage{hyperref}
\usepackage{stackrel}

\newtheorem{thm}{Theorem}[section]
\newtheorem{prop}[thm]{Proposition}

\newtheorem{lem}[thm]{Lemma}

\newtheorem{cor}[thm]{Corollary}

\newtheorem{que}[thm]{Question}
\theoremstyle{definition}
\newtheorem{obs}[thm]{Observation}

\newtheorem{rem}[thm]{Remark}
\newtheorem{exa}[thm]{Example}
\newtheorem{defn}[thm]{Definition}

\newtheorem{alg}[thm]{Algorithm}

\newcommand{\sumlim}{\sum\limits}

\newcommand{\dep}{{\rm dp}}
\newcommand{\Neg}{{\rm Neg}}
\newcommand{\Inv}{{\rm Inv}}
\newcommand{\nneg}{{\rm neg}}
\newcommand{\inv}{{\rm inv}}
\newcommand{\sgn}{{\rm sign}}
\newcommand{\nsp}{{\rm nsp}}
\newcommand{\Nsp}{{\rm Nsp}}
\begin{document}

\title[Depth in classical Coxeter groups]{Depth in classical Coxeter groups}

\author{Eli Bagno}
\email{bagnoe@g.jct.ac.il}
\address{Jerusalem College of Technology, 21 Havaad Haleumi st. Jerusalem, Israel.}
\author{Riccardo Biagioli}
\email{biagioli@math.univ-lyon1.fr}
\address{Institut Camille Jordan, Universit\'e Claude Bernard Lyon 1, 69622 Villeurbanne Cedex, France.}
\author{Mordechai Novick}
\email{mnovick@g.jct.ac.il}
\address{Jerusalem College of Technology, 21 Havaad Haleumi st. Jerusalem, Israel.}
\author{Alexander Woo}
\thanks{AW is partially supported by NSA Young Investigators Grant H98230-13-1-0242}
\email{awoo@uidaho.edu}
\address{Dept. of Mathematics, University of Idaho, 875 Perimeter Dr. MS 1103, Moscow, ID 83844-1103, USA.}


\keywords{Coxeter groups, Bruhat graph, reflections, depths, length.}


\begin{abstract}
The depth statistic was defined by Petersen and Tenner for an element of an arbitrary Coxeter group in terms of factorizations of the element into a product of reflections. It can also be defined as the minimal cost, given certain prescribed edge weights, for a path in the Bruhat graph from the identity to an element. We present algorithms for calculating the depth of an element of a classical Coxeter group that yield simple formulas for this statistic. We use our algorithms to characterize elements having depth equal to length. These are the short-braid-avoiding elements. We also give a characterization of the elements for which the reflection length coincides with both the depth and the length. These are the boolean elements.
\end{abstract}


\maketitle

\section{Introduction}
\label{sec:intro}

Let $(W,S)$ be a Coxeter system.  Two statistics on elements of $W$ are classically associated to any Coxeter system.  First, each element $w\in W$ can be written as a product $w=s_1\cdots s_r$ of simple generators $s_i\in S$.  The {\bf length} $\ell_S(w)$ of $w$ is the minimal number of simple generators $s\in S$ needed to express $w$, so
\begin{equation}\label{length}
\ell_S(w):=\mbox{min} \{r \in {\mathbb N} \mid w=s_{1}\cdots s_{r}\; \;
\mbox{for some} \; \; s_{1},\ldots,s_{r} \in S \}.
\end{equation}
If $r=\ell_S(w)$, then any word $w=s_1\cdots s_r$ is called a {\bf reduced expression} for $w$.

The second statistic is known as reflection length.  Let
\begin{equation}\label{reflections}
T:=\{wsw^{-1} \mid s \in S, \; w \in W \}.
\end{equation}
This is known as the set of {\bf reflections} of $W$.  The {\bf reflection length} $\ell_T(w)$ is the minimal number of reflections $t\in T$ needed to express $w$, so
\begin{equation}\label{reflength}
\ell_T(w):=\mbox{min} \{r \in {\mathbb N} \mid w=t_{1}\cdots t_{r}\; \;
\mbox{for some} \; \; t_{1},\ldots,t_{r} \in T \}.
\end{equation}

Let $\Phi=\Phi^+ \cup \Phi^{-}$ be the root system for $(W,S)$, with $\Pi \subset \Phi$ the simple roots. The depth $\dep(\beta)$ of a positive root $\beta \in \Phi^+$ is defined as  
$$\dep(\beta):={\rm min}\{r \mid s_{1} \cdots s_{r} (\beta) \in  \Phi^-, s_j \in S\}.$$
It is easy to see that $\dep(\beta)=1$ if and only if $\beta \in \Pi$. As a function on the set of roots, depth is also the rank function for the root poset of a Coxeter group, as developed in \cite[\S 4]{BBbook}.

Now, if we denote by $t_{\beta}$ the reflection corresponding to the root $\beta$, Petersen and Tenner introduced~\cite{PT} a new statistic, also called {\bf depth} and denoted $\dep(w)$ for any $w \in W$, by defining
$$\dep(w):={\rm min} \left\{\sum_{i=1}^r{\dep(\beta_i)} \mid w=t_{\beta_1} \cdots t_{\beta_r}, \ t_{\beta_i} \in T\right\}.$$

Petersen and Tenner further observe that depth always lies between length and reflection length.  For each positive root $\beta$, one has
\begin{equation}\label{eq:depthreflection}
\dep(t_{\beta})=\dep(\beta)=\frac{\ell_S(t_\beta)+1}{2}.
\end{equation}
Hence, by definition,
$$\ell_T(w)\leq \frac{\ell_T(w)+\ell_S(w)}{2} \leq \dep(w) \leq \ell_S(w).$$

Petersen and Tenner focus mainly on the case where $(W,S)$ is the symmetric group (with $S$ being the adjacent transpositions).
In particular, they provide the following:
\begin{itemize}
\item A formula for depth in terms of the sizes of exceedances.  To be specific, they show
$$\dep(w)=\sumlim_{w(i)>i} (w(i)-i).$$  
\item The maximum depth for an element in $S_n$ (for each fixed $n$) and a characterization of the permutations that achieve this depth. (Both were also previously found by Diaconis and Graham \cite{DiGr} starting from the above formula, which they called the total displacement of a permutation.)
\item An algorithm that, given $w\in W$, finds an expression $w=t_1\cdots t_r$ that realizes the depth of $w$.
\item Characterizations both of the permutations $w$ such that $\dep(w)=\ell_S(w)$ and of the permutations such that $\dep(w)=\ell_T(w)$.
\end{itemize}

In this paper, we provide analogous results for the other infinite families of finite Coxeter groups, namely the group $B_n$ of signed permutations and its subgroup $D_n$ of even signed permutations.  (The dihedral groups were also treated in~\cite{PT}.)  Definitions and our conventions for working with these groups follow in Section 2.

In each case, we give a formula for depth that, like those in~\cite{PT}, is in terms of sizes of exceedances, except we need to introduce a small adjustment factor that can be explicitly calculated from the interaction of the signs and the sum decomposition of the underlying unsigned permutation.  Using our formulas, we can find the maximum depth for any element in $B_n$ or $D_n$ for a given $n$ and describe the signed permutations that achieve this maximum.

Furthermore, we give algorithms that, given an element $w$, produce reflections $t_1,\ldots, t_r$ such that $w=t_1\cdots t_r$ and $\dep(w)=\sum_{i=1}^r \dep(t_r)$.  Our algorithms differ from that of Petersen and Tenner in several respects that are worth mentioning.  First, their algorithm produces an expression with $r=\ell_T(w)$ reflections.  This turns out to be impossible in types B and D; indeed there are elements $w$ in both $B_n$ and $D_n$ where no factorization of $w$ into $\ell_T(w)$ reflections realizes the depth.  Second, both their algorithm and ours always produce an expression that is realized by a strictly increasing path in Bruhat order.  To be precise, if we let $w_i=t_1\cdots t_i$ for all $i$ with $1\leq i\leq r$, then $\ell_S(w_i)\geq \ell_S(w_{i-1})$ for all $i$, $2\leq i\leq r$.  However, our algorithms have a stronger property; we always produce a factorization for which $\ell_S(w_i)=\ell_S(w_{i-1})+\ell_S(t_i)$, and hence $\ell_S(w)=\sum_{i=1}^r \ell_S(t_i)$.  In other words, our factorization of $w$ into reflections is reduced.  Finally, the algorithm of Petersen and Tenner relies on both left and right multiplication.  Given a permutation $w$, their algorithm produces either the reflection $t_1$ or the reflection $t_r$, and then their algorithm recursively finds a decomposition for either $t_1w$ or $wt_r$ respectively.  Our algorithms use only right multiplication.  Given $w$, we always produce the reflection $t_r$ and then recursively apply our algorithm to $wt_r$.
In particular, this implies that the depth can be realized by a strictly increasing path in the {\bf right weak order}, which, as we recall, is the partial order on $W$ that is the transitive closure of the relation where, for any $w\in W$ and $s\in S$, $w <_R ws$ if $\ell_S(w)+1=\ell_S(ws)$.  (See~\cite[\S 3.1]{BBbook} for further details.)

Using the property that depth is always realized by a reduced factorization into reflections, it is easy to see that the elements $w$ with $\ell_S(w)=\dep(w)$ are precisely the short-braid-avoiding elements introduced by Fan~\cite{Fan}.  In type B, these are precisely the top-and-bottom fully commutative elements of Stembridge~\cite{St2}, confirming the conjecture of Petersen and Tenner.  The elements $w$ with $\ell_T(w)=\dep(w)$ are the boolean elements of Tenner~\cite{Tenner1}.

The remainder of our paper is organized as follows.  Section \ref{sec:def} gives our conventions for type B and type D and states our main theorems.  Section \ref{sec:typeBalg} gives our algorithm and the proof of our formula in type B, and Section \ref{sec:typeDalg} does the same for type D.  Section \ref{sec:coincidences} discusses the coincidences of depth, length, and reflection length. Section \ref{sec:open} lists a series of open problems and  various considerations stemming from the realization of depth by reduced factorizations.
\smallskip

We thank Kyle Petersen for useful discussions, including letting us know that we were unnecessarily duplicating efforts by independently working on these questions.

\section{Definitions, Conventions, and Main Results}\label{sec:def}

Unless otherwise stated, the conventions we use for working with $B_n$ and $D_n$ are those of~\cite[\S 8.1--8.2]{BBbook}, where more details and proofs of statements can be found. For $n \in {\mathbb N}$ we set $[n]:=\{1,\ldots,n\}$, and for $m,n \in \mathbb{Z}$ we set $[m,n]:=\{m,m+1,\ldots, n\}$.

\subsection{Type B}
We define $B_n$ to be the group of {\bf signed permutations} on $[n]$. The is the group of all bijections $w$ of the set $[-n,n]\setminus\{0\}$ onto itself such that $w(-i)=-w(i)$ for all $i\in[-n,n]\setminus\{0\}$, with composition as the group operation.  Since an element $w\in B_n$ is entirely determined by $w(i)$ for $i>0$, we write $w=[w(1),\ldots,w(n)]$ and call this the {\bf window notation} for $w$.  For convenience, we will usually write our negative signs above a number rather than to its left when using window notation.  We denote by $\sgn(w(i))$ and $|w(i)|$ the {\bf sign} and {\bf absolute value} of the entry $w(i)$.  We let
\begin{eqnarray*}
s_0^B&:=&[\overline{1},2,\ldots,n]=(-1,1), \\
s_i&:=& [1,\ldots,i-1,i+1,i,i+2,\ldots,n]=(i,i+1)(-i,-i-1) \quad {\rm for} \ i \in [n-1],
\end{eqnarray*}
and $S_B:=\{s_0^B,s_1,\ldots,s_{n-1}\}$.  Then $(B_n,S_B)$ forms a Coxeter system.
\smallskip

To state our formula for depth, we need the notion of an indecomposable element and some associated definitions.  These are standard definitions for permutations, but as far as we are aware, they have not been previously extended to signed permutations.

\begin{defn}\rm
Let $u \in B_k, v \in B_{n-k}$. Define the {\bf direct sum} of $u$ and $v$ by:
$$(u \oplus v)(i) :=\left\{ \begin{array}{lll} u(i) & i \in \{1,\ldots, k\};  \\
\sgn(v(i-k)) (|v(i-k)|+k) & i \in \{k+1,\ldots, k+l\}.
\end{array}    \right.$$
\end{defn}

A signed permutation $w \in B_n$ is {\bf decomposable} if it can be expressed as a nontrivial (meaning $1\leq k\leq n-1$) direct sum of signed permutations and {\bf indecomposable} otherwise.  Every signed permutation has a unique expression as the direct sum of indecomposable signed permutations $w=w^1 \oplus \cdots \oplus w^k$.  This expression is called the {\bf type B decomposition} of $w$. The indecomposable pieces are called the {\bf type B blocks} (or simply {\bf blocks}).

\begin{defn}
Given a signed permutation $w=w^1 \oplus \cdots \oplus w^k \in B_n$, we define the {\bf B-oddness} of $w$, denoted by $o^B(w)$, as the number of blocks in the sum decomposition with an odd number of negative entries. 
\end{defn}

For example, if we let $w=[4, \bar{3}, 1, \bar{2}, 7, 5, \bar{6}, 9, \bar{8}]$, then $w$ is decomposable with $w=w^1 \oplus  w^2 \oplus w^3$, where the blocks are
$w^1=[4, \bar{3}, 1, \bar{2}]$, $w^2=[3, 1, \bar{2}]$, and $w^3=[2,\bar{1}]$; moreover $o^B(w)=2$.  On the other hand, $w'=[\bar{8}, 1, 9, 3, 5, 2, \bar{6}, 4, 7]$ is indecomposable with $o^B(w')=0$. The negative identity $[\bar1,\ldots,\bar n]$ is the oddest element in $B_n$, with oddness $n$.

Given a signed permutation $w$, let 
$$ \Neg(w):=\{i\in[n]\mid w(i)<0\} \quad {\rm and} \quad \nneg(w):=\#\Neg(w).$$  
Now we can present an explicit formula for the depth of a permutation in $B_n$ in terms of window notation.

\begin{thm}\label{thm:depthB}
Let $w \in B_n$. Then
\begin{equation}\label{eq:d_B}
\dep(w)=\sum\limits_{\{i \in [n] \mid w(i)>i\}}{(w(i)-i)} + \sum\limits_{i \in \Neg(w)}{|w(i)|} +\frac{o^B(w)-\nneg(w)}{2}.
\end{equation}
\end{thm}

One can also reformulate this formula as follows to more closely match the formula of Petersen and Tenner for the symmetric group:

\begin{cor}\label{cor:depthBalt}
Let $w\in B_n$.  Then
$$\dep(w)=\frac{1}{2}\left(\left(\sum\limits_{\{i\in[-n,n]\setminus\{0\}\mid w(i)>i\}} {(w(i)-i)}\right)+o^B(w)-\nneg(w)\right).$$
\end{cor}

The group $B_n$ can also be realized as the subgroup of the symmetric group $S_{2n}$ consisting of the permutations $w$ with $w(i)+w(2n+1-i)=2n+1$ for all $i\in[n]$.  Under this realization of $B_n$, the formula will translate to a similar form, but the term $\nneg$ will not appear any more.

Using our formula, we can easily show:
\begin{cor}\label{cor:maxdepthB}
For each $w \in B_n$ we have $\dep(w) \leq {n+1 \choose 2}$, with equality if and only if $w= [\bar{1}, \bar{2}, \ldots, \bar{n}]$.
\end{cor}

Petersen and Tenner ask if $\dep(w)$ can always be realized by a product of $\ell_T(w)$ reflections.  The following example shows that this is impossible in type B.

\begin{exa}\label{Brefllengthcounterex}
Let $w=[\bar{4},\bar{2},\bar{3},\bar{1}] \in B_4$.  Then $\dep(w)=8$, since $w$ is indecomposable and $o^B(w)=0$.  However, $\ell_T(w)=3$, and there are essentially only two ways to write $w$ as the product of 3 reflections.  One is to write $w$ as the product of $t_{\bar{1} 4}=[\bar{4},2,3,\bar{1}]$, $t_{\bar{2}2}=[1,\bar{2},3,4]$, and $t_{\bar{3}3}=[1,2,\bar{3},4]$ in some order.  (These reflections pairwise commute.)  The sum of the depths of these reflections is $9>8$.  One can also write $w$ as the product of $t_{\bar{1}4}=[\bar{4},2,3,\bar{1}]$, $t_{\bar{2}3}=[1,\bar{3},\bar{2},4]$, and $t_{23}=[1,3,2,4]$ in some order.  The sum of the depths of these reflections is also $9>8$.
\end{exa}

However, we will show that our algorithm always produces a factorization of $w$ with the following property.

\begin{thm}\label{thm:typeBred}
Let $w\in B_n$.  Then there exist reflections $t_1,\ldots,t_r$ such that $w=t_1\cdots t_r$, $\dep(w)=\sum_{i=1}^r \dep(t_i)$, and $\ell_S(w)=\sum_{i=1}^r \ell_S(t_i)$.
\end{thm}
One says that $w=t_1\cdots t_r$ is a {\bf reduced  factorization} if $\ell_S(w)=\sum_{i=1}^r \ell_S(t_i)$.  This theorem says that the depth of $w$ is always realized by a reduced factorization of $w$ into transpositions.  Note that we can consider $S_n$ as the subgroup of $B_n$ consisting of permutations with no negative signs or equivalently as the Coxeter subgroup generated by $s_1,\ldots,s_{n-1}$.  Hence this theorem holds for $S_n$, and it is new even in that case.

\subsection{Type D}
We define $D_n$ to be the subgroup of $B_n$ consisting of signed permutations with an even number of negative entries when written in window notation, or, more precisely, we define $D_n:=\{w\in B_n\mid \nneg(w)\mathrm{ \ is \ even}\}.$
Let
\begin{eqnarray*}
s_0^D&:=&[\bar{2},\bar{1},3,\ldots,n]=(1,-2)(-1,2), 
\end{eqnarray*}
and $S_D=\{s_0^D,s_1,\ldots,s_{n-1}\}$, where the $s_i$'s are defined as in type B for $i\in[n-1]$. Then $(D_n,S_D)$ forms a Coxeter system.

To state our formula in type D, we first need to be more careful about our notion of decomposibility.  Given a signed permutation $w\in D_n$, we can give a decomposition of $w$ as $w=w^1\oplus\cdots\oplus w^k$, where we insist that each $w^i\in D_m$, $m\leq n$ and, furthermore, no $w^i$ is a direct sum of elements of $D_p$, $p<m$.  We call this decomposition of $w$ a {\bf type D decomposition} and the blocks of this decomposition {\bf type D blocks}.  

We can also look at $w \in D_n$ as an element of $B_n$ and consider its type B decomposition. Note that a type D block $w^i$ may split into $b_i$ smaller B blocks, which we denote $w^i=w^i_1\oplus\cdots \oplus w^i_{b_i}$, where possibly $b_i=1$. Note that, whenever $b_i>1$, $w_1^i$ and $w_{b_i}^i$ must have an odd number of negative entries and the remaining central B blocks $w^i_2,\ldots,w^i_{b_i-1}$ must have an even number of negative entries.

\begin{defn}
For each $w \in D_n$, we define the {\bf D-oddness} of $w$, denoted by $o^D(w)$, to be the difference between the number of type B blocks and the number of type D blocks, or, equivalently, we define $o^D(w):=\sum_i (b_i-1)$.
\end{defn}

For example, if $w=[-2,1,3,4,-5,-7,-8,6]$, then the type D decomposition is $w=w^{1}\oplus w^2$ where $w^1=[-2,1,3,4,-5]$ and $w^2=[-2,-3,1]$, while the type B decomposition has $w=w^1_1\oplus w^1_2\oplus w^1_3\oplus w^1_4\oplus w^2_1$ (so $b_1=4$ and $b_2=1$) with $w^1_1=[-2,1]$, $w^1_2=[1]$, $w^1_3=[1]$, $w^1_4=[-1]$, and $w^2_1=w^2=[-2,-3,1]$. Hence $o^D(w)=3$. The oddest element in $D_n$ is $[\bar{1},2,\ldots,n-1,\bar{n}]$, with D-oddness $n-1$.
\smallskip

Since every B-decomposable type D block has exactly two type B blocks with an odd number of negative entries, and the D-oddness of $w$ is at least its number of B-decomposable type D blocks, we have $o^D(w)\geq\frac{1}{2}o^B(w)$ (where $o^B(w)$ is calculated by considering $w$ as an element of $B_n$ via the embedding of $D_n$ in $B_n$).

Now we can state our formula for depth in type D.  Note that depth depends on the Coxeter system, so $w\in D_n$ will have different depth considered as an element of $D_n$ compared to considering it as an element of $B_n$.

\begin{thm}\label{thm:depthD}
Let $w \in D_n$. Then
\begin{equation}\label{eq:d_D}
\dep(w)=\left(\sum\limits_{\{i \in [n] \mid w(i)>i\}}{(w(i)-i)}\right) + \left(\sum\limits_{i \in \Neg(w)}{|w(i)|} \right)+(o^D(w)-\nneg(w)).
\end{equation}
\end{thm}

One also has a reformulation more closely matching the formula of Petersen and Tenner for the symmetric group:

\begin{cor}\label{cor:depthDalt}
Let $w\in D_n$.  Then
$$\dep(w)=\frac{1}{2}\left(\left(\sum\limits_{\{i\in[-n,n]\setminus\{0\}\mid w(i)>i\}} {(w(i)-i)}\right)-2 \! \ \nneg(w)\right)+o^D(w).$$
\end{cor}

Using our formula, we can show the following:
\begin{cor}\label{cor:maxdepthD}
For each $w \in D_n$ we have $\dep(w)\leq {n \choose 2} + \left \lfloor {\frac{n}{2}} \right \rfloor$.  Equality occurs for $2^{\frac{n-2}{2}}$ elements if $n$ is even and $2^{\frac{n+1}{2}}$ elements if $n$ is odd.
\end{cor}

The example in type B showing that $\dep(w)$ cannot always be realized by a product of $\ell_T(w)$ reflections also works in type D (though the depths are different).

\begin{exa}\label{Drefllengthcounterex}
Let $w=[\bar{4},\bar{2},\bar{3},\bar{1}]\in D_4$.  Then $\dep(w)=6$, since $w$ is indecomposable and $o^D(w)=0$.  However, $\ell_T(w)=3$, and the only ways to write $w$ as the product of 3 reflections are as the product of $t_{\bar{1}4}=[\bar{4},2,3,\bar{1}]$, $t_{\bar{2}3}=[1,\bar{3},\bar{2},4]$, and $t_{\bar{3}3}=[1,3,2,4]$ in some order.  (These reflections pairwise commute.)  The sum of the depths of these reflections is $7>6$.
\end{exa}

As in type B, however, our algorithm always produces a reduced factorization of $w$.

\begin{thm}\label{thm:typeDred}
Let $w\in D_n$.  Then there exist reflections $t_1,\ldots,t_r$ such that $w=t_1\cdots t_r$, $\dep(w)=\sumlim_{i=1}^r \dep(t_i)$, and $\ell_S(w)=\sumlim_{i=1}^r \ell_S(t_i)$.
\end{thm}

\subsection{Coincidences of depth, length, and reflection length}
The proofs of Theorems~\ref{thm:typeBred} and~\ref{thm:typeDred} are distinct, each using the specific combinatorial realization of these groups described above.  However, using these theorems, we can uniformly prove results on the coincidence of depth, length, and reflection length.

Following Fan~\cite{Fan}, we say that $w\in W$ is {\bf short-braid-avoiding} if there does not exist a consecutive subexpression $s_i s_j s_i$, where $s_i, s_j\in S$, in any reduced expression for $w$.

\begin{thm}
Let $w$ be an element of $S_n$, $B_n$, or $D_n$.  Then $\dep(w)=\ell_S(w)$ if and only if $w$ is short-braid-avoiding.
\end{thm}

Since in $B_n$, the short-braid-avoiding elements are precisely the top-and-bottom fully commutative elements defined by Stembridge~\cite[\S 5]{St2}, we confirm a conjecture of Petersen and Tenner~\cite[\S 5]{PT}.

Let $W$ be any Coxeter group.  An element $w \in W$ is called {\bf boolean}
if the principal order ideal of $w$ in $W$, $B(w):=\{x \in W \mid x\leq w\}$, is a boolean poset, where $\leq$ refers to the strong Bruhat order. Recall that a poset is boolean if it is isomorphic to the set of subsets of $[k]$ ordered by inclusion for some $k$.

\begin{thm}
Let $W$ be any Coxeter group and $w\in W$.  Then $\dep(w)=\ell_T(w)$ if and only if $w$ is boolean.
\end{thm}

These lead to the following enumerative results.

\begin{cor}\label{num dp=l}\
\begin{enumerate}
\item The number of elements $w \in B_n$ satisfying $\dep(w)=\ell_S(w)$ is the Catalan number ${\rm C}_{n+1}$.
\item The number of elements $w \in D_n$ satisfying $\dep(w)=\ell_S(w)$ is $\frac{1}{2}(n+3) {\rm C}_n-1$.
\end{enumerate}
\end{cor}

\begin{cor}\label{thm:enumeration}\
\begin{enumerate}
\item The number of elements $w \in B_n$ satisfying $\ell_{T}(w)=\dep(w)=\ell_S(w)$ is the Fibonacci number $F_{2n+1}$.
\item The number of elements $w \in B_n$ satisfying $\ell_{T}(w)=\dep(w)=\ell_S(w)=k$ is
$$\sumlim_{i=1}^k{{{n+1-i}\choose {k+1-i}}{{k-1}\choose {i-1}}},$$
where for $k=0$ the sum is defined to be 1.
\end{enumerate}
\end{cor}

\begin{cor}\label{thm:enumerationD}\
\begin{enumerate}
\item For $n\geq 4$, the number of elements $w \in D_n$ satisfying $\ell_{T}(w)=\dep(w)=\ell_S(w)$ is 
$$\frac{13-4b}{a^2(a - b)} a^n + \frac{13 - 4a}{b^2(b - a)}b^n,$$
where $a=(3+ \sqrt 5)/2$ and $b=(3- \sqrt 5)/2$.
\item For $n > 1$, the number of elements $w \in D_n$ satisfying $\ell_{T}(w)=\dep(w)=\ell_S(w)=k$ is
$$\quad \quad L^D(n, k) = L(n, k) + 2L(n, k - 1) - L(n - 2, k- 1) - L(n - 2, k - 2),$$	
where $L(n,k)=\sumlim_{i=1}^k {n-i \choose k+1-i}{k-1 \choose i-1}$, $L(n,k)$ is 0 for any $(n,k)$ on which it is undefined, and $L^D(1,0) = 1$ and $L^D(1,1) = 0$.
\end{enumerate}
\end{cor}


\section{Realizing depth in type B}
\label{sec:typeBalg}

\subsection{Reflections, length, and depth in type B}

For the reader's convenience we state here the basic facts on reflections, length, and depth for the Coxeter system $(B_n,S_B)$.  The facts on reflections and their lengths can be found in~\cite[\S 8.1]{BBbook}.

The set of reflections is given by
$$T^B:=\{t_{ij}, t_{\bar{i}j}   \mid 1 \leq i < j \leq n\} \cup \{t_{\bar{i}i} \mid i \in [n]\},$$
where $t_{ij}=(i,j)(\bar{i},\bar{j})$,   $t_{\bar{i}j}=(\bar{i},j)(i,\bar{j})$, and $t_{\bar{i}i}=(\bar{i},i)$ in cycle notation. In particular there are $n^2$ reflections in $B_n$.  We summarize below the result in window notation of right multiplication by each type of reflection.

\begin{enumerate}
\item The reflections $t_{ij}$.

\noindent Right multiplication of $w$ by  $t_{ij}$ swaps the entry $w(i)$ in location $i$ with the entry $w(j)$ in location $j$ in such a way that each digit moves with its sign. For example $ [\bar{3}, 1, 4, 2] \stackrel{t_{12}}{\rightarrow}  [1, \bar{3}, 4, 2].$

\item The reflections $t_{\bar{i}j}$.

\noindent Right multiplication by $t_{\bar{i}j}$ swaps entry $w(i)$ in location $i$ with entry $w(j)$ in location $j$ and  changes both signs. For example
$ [\bar{3}, 1, 4, 2] \stackrel{t_{\bar{1}2}}{\rightarrow} [\bar{1}, 3, 4, 2].$

\item The reflections $t_{\bar{i}i}$.

\noindent Right multiplication by $t_{\bar{i}i}$ changes the sign of the entry $w(i)$ in location $i$. For example $[3, \bar{1}, 4, 2]
\stackrel{t_{\bar{2}2}}{\rightarrow}  [3, 1, 4, 2]$.
\end{enumerate}

The length of a permutation $w \in B_n$ is measured by counting certain pairs of entries, known as B-inversions. We can divide them into three types. For $w \in B_n$, we have already defined the set ${\rm Neg}(w)$; now we set
\begin{itemize}
\item[] ${\rm Inv}(w):=\{(i,j) \mid 1 \leq i<j \leq n, w(i)>w(j)\},$ and

\item[] ${\rm Nsp}(w):=\{(i,j) \mid 1\leq i<j\leq n, w(i)+w(j)<0\}.$
\end{itemize}

If we let ${\rm inv}(w):=\#{\rm Inv}(w)$ and ${\rm nsp}(w):=\#{\rm Nsp}(w)$, then we have the following  formula for the length
\begin{equation}\label{eq:typeBlength}
\ell_S(w)={\rm inv}(w)+{\rm neg}(w)+{\rm nsp}(w).
\end{equation}

Note that a pair $(i,j)$ may appear in both ${\rm Inv}(w)$ and ${\rm Nsp}(w)$, in which case this pair is counted twice in calculating $\ell_S(w)$.

From Equations (\ref{eq:depthreflection}) and (\ref{eq:typeBlength}) we immediately obtain the depths of the three types of reflections

\begin{lem}\label{lem:typeBdepthreflections}
The depths of reflections in type B are as follows.
$$\dep(t_{ij})=j-i, \quad \dep(t_{\bar{i}j})=i+j-1, \quad and \quad \dep(t_{\bar{i}i})=i.$$
\end{lem}

We can also easily determine if $w=vt$ is a reduced factorization.

\begin{lem}\label{lem:typeBreflprodred}
Let $v, w\in B_n$, $t\in T^B$, and $w=vt$.  Then $\ell_S(w)=\ell_S(v)+\ell_S(t)$ if and only if one of the following hold:
\begin{enumerate}
\item $t=t_{ij}$, $w(i)>w(j)$, and for all $k$ with $i<k<j$, $w(i)>w(k)>w(j)$.
\item $t=t_{\bar{i}j}$, $w(i)<0$, $w(j)<0$, for all $k$ with $1\leq k<i$, $w(k)>w(i)$ and $w(k)+w(j)<0$, and for all $k'\neq i$ with $1\leq k'<j$,  $w(k')>w(j)$ and $w(i)+w(k')< 0$.
\item $t=t_{\bar{i}i}$, $w(i)<0$, and for all $k$ with $1\leq k<i$, we have $|w(k)|<|w(i)|$.
\end{enumerate}
\end{lem}

\begin{proof}

Let $t$ be equal to $t_{ij}$, $t_{\bar{i}j}$ or $t_{\bar{i}i}$, and let $v,w \in B_n$ such that $w=vt$.
It is obvious that $\ell_S(w) \leq \ell_S(v)+\ell_S(t)$, so we only need to determine when strict inequality occurs.  We write each $t$ as a product of simple reflections $t=s_{i_1} \cdots s_{i_r}$ and show that each of the conditions of the lemma corresponds to the assertion that, for each $k \in [r]$, one has $\ell_S(ws_{i_1} \cdots s_{i_{k-1}} s_{i_k})<\ell_S(ws_{i_1} \cdots s_{i_{k-1}})$.

\begin{enumerate}
\item 

If $t=t_{ij}$, then $t=s_{j-1}s_{j-2} \cdots s_i s_{i+1} \cdots s_{j-1}$.  
Each appearance of $s_k$ in this expression ($i\leq k<j)$ has the effect of exchanging $w(k)$ with either $w(i)$ or $w(j)$.  This reduces the length for each $s_k$ if and only if condition (1) is met. 

\item If $t=t_{\bar{i}j}$ then its decomposition into simple reflections is $$t=s_{i-1}\cdots s_1  s_{j-1} \cdots s_2 s_0^B s_1 s_0^B s_2 \cdots s_{j-1}s_1 s_2 \cdots s_{i-1}.$$ 

In the first consecutive substring $s_{i-1}\cdots s_1$, multiplying by $s_k$ reduces the length (for $1 \leq k \leq 
i-1$) if and only if $w(i)<w(k)$. 
Similarly, in the second part $s_{j-1}\cdots s_2$, multiplying by $s_{k'}$ (for $i+1\leq k'<j$) or $s_{k'+1}$ (for $1\leq k'<i$) reduces the length if and only if $w(j)<w(k)$.

Then, in the next consecutive substring $s_0^Bs_1s_0^B$, the two appearances of $s_0^B$ reduce length  if and only if $w(i),w(j)<0$. Furthermore, whenever we have $w(i),w(j)<0$, the intervening $s_1$ reduces length since it moves a negative entry to the left of a positive entry.

The following part $s_2 \cdots s_{j-1}$ moves the entry $-w(i)$ to the $j$-th position. Hence, each $s_{k'+1}$ (if $1\leq k'<i$) or $s_{k'}$ (if $i+1\leq k'<j$) will reduce length if and only if $-w(i)>w(k)$. 
Similarly, each $s_{k}$ in the final substring $s_1 \cdots s_{i-1}$ reduces length if and only if $-w(j)>w(k)$ for each $1 \leq k <i$. 

\item 
If $t=t_{\bar{i}i}$ then $t=s_is_{i-1} \cdots s_1 s_0^B s_1 s_2 \cdots s_i$. Here, each application of $s_k$ appearing before $s_0^B$ reduces the length if and only if $w(k)>w(i)$. After the application of $s_0^B$, each successive $s_k$ will reduce length if and only if $-w(i)>w(k)$.  Hence, in order for every $s_k$ in the word representing $t$ to reduce length, it is necessary and sufficient that $|w(k)|<|w(i)|$ for all $1\leq k<i$.  (The single $s_0^B$ clearly reduces length if and only if $w(i)<0$.)

\end{enumerate}
\end{proof}

\subsection{Algorithm for realizing depth in type B}

To keep the logical status of our theorem clear as we work through the proof, we let $d(w)$ denote our expected value of $\dep(w)$ according to our formula.  Hence, for $w\in B_n$, let
\begin{equation*}
d(w):=\sum\limits_{\{i \in [n] \mid w(i)>i\}}{(w(i)-i)} + \sumlim_{i \in \Neg(w)}{|w(i)|}+\frac{o^B(w)-\nneg(w)}{2}.
\end{equation*}

In order to prove Theorem \ref{thm:depthB}, we proceed in two steps. First, we supply an algorithm that associates to each $w \in B_n$  a decomposition of $w$ into a product of reflections whose sum of depths is $d(w)$. This will prove that $d(w)\geq\dep(w)$. Then we will show that $d(w)\leq\dep(w)$.

It will be more convenient for us to describe our algorithm as an algorithm to sort $w$ to the identity signed permutation $[1,\ldots,n]$ using a sequence of right multiplications by reflections $t_{i_k j_k}$.  One can then recover a decomposition of $w$ as the product of these reflections in reverse order.  Our algorithm is applied on each indecomposable part of $w$ separately; hence from now on we assume that $w$ is indecomposable.  We say that an entry $x$ is in its {\bf natural position} in $w$ if $x=w(x)$.

\begin{alg} Let $w \in B_n$ be indecomposable.
\begin{enumerate}
\item Let $i$ be the position such that the entry $w(i)$ is maximal among all $w(i)$ with $w(i)>i$, for $i\in[n]$.  (If no such $i$ exists, continue to the next step.)  Now let $j$ be the index such that the entry $w(j)$ is minimal among all $j$ with $i<j\leq w(i)$.  Right multiply by $t_{ij}$.  Repeat this step until $w(i)\leq i$ for all $i$, $1\leq i\leq n$.

\noindent Let $u$ be the element obtained after the last application of Step 1, $ w \stackrel{t_{i_1j_1}}{\rightarrow} \cdots \stackrel{t_{i_kj_k}}{\rightarrow} u.$

\item If $\nneg(u)\geq2$ then right multiply $u$ by $t_{\bar{ij}}$, where $w(i)$ and $w(j)$ are the two negative entries of largest absolute value in $u$, and go back to Step 1.  If $\nneg(u)=1$ then right multiply $u$ by $t_{\bar{i}i}$, where $w(i)$ is the sole negative entry, and go back to Step 1.  If $\nneg(u)=0$, then we are finished.
\end{enumerate}
\end{alg}

In other words, the algorithm begins by shuffling each positive entry $w(i)$ which appears to the left of its natural position into its natural position, starting from the largest and continuing in descending order.
Once this is completed, an unsigning move is performed.  If there is more than one negative entry in $w$, we unsign a pair, thus obtaining two new positive entries. The process restarts, and the entries may be further shuffled.
Unsigning and shuffling moves continue to alternate until neither type of move can be performed. The last unsigning move will be a single one if the number of negative entries in $w$ is odd.

At the end of the algorithm, there are no negative entries, and no positive entry is to the left of its natural position.  Hence we must have the identity signed permutation.

Note that although the algorithm assumes that $w$ is indecomposable, it can happen in the course of the algorithm that $w$ is transformed into a decomposable permutation. This does not pose a problem since the only way this occurs is by creating blocks on the right consisting entirely of positive entries that are further acted upon only by Step 1, and indecomposability matters only in determining when $\nneg(u)=1$ in Step 2.

\smallskip

We demonstrate our algorithm in the following example.
The depth of each reflection is given below the corresponding arrow.

\begin{exa}

Let $w=[\bar{6}, \bar{3}, \bar{2}, 8, 7, 5, 9, \bar{4}, \bar{1}] \in B_9$. Our first step is to shuffle entry $9$ to position 9:
$$w= [\bar{6}, \bar{3}, \bar{2}, 8, 7, 5, {\bf 9}, \bar{4}, \bar{1}]
\stackrel[1]{t_{78}}{\rightarrow}
[\bar{6}, \bar{3} , \bar{2}, 8, 7, 5, \bar{4}, {\bf 9}, \bar{1}]
\stackrel[1]{t_{89}}{\rightarrow}
[\bar{6}, \bar{3}, \bar{2},  8, 7, 5, \bar{4}, \bar{1}, 9].$$
Then we further apply Step 1, first to the entry 8 and then to 7:
\begin{eqnarray*}
[\bar{6}, \bar{3}, \bar{2}, {\bf 8}, 7, 5, \bar{4}, \bar{1}, 9] \stackrel[3]{t_{47}}{\rightarrow} [\bar{6}, \bar{3}, \bar{2}, \bar{4}, 7, 5, {\bf 8}, \bar{1}, 9] &\stackrel[1]{t_{78}}{\rightarrow}& [\bar{6}, \bar{3}, \bar{2}, \bar{4}, {\bf 7}, 5, \bar{1}, 8,  9]\\
&\stackrel[2]{t_{57}}{\rightarrow}& [\bar{6}, \bar{3}, \bar{2}, \bar{4},  \bar{1}, 5, 7,  8,  9].
\end{eqnarray*}
Now, none of the positive entries are located to the left of its natural position, so we proceed with Step 2 to unsign the two most negative entries, which are $\bar{6}$ and $\bar{4}$:
$$[\bar{\bf 6}, \bar{3}, \bar{2}, \bar{\bf 4},  \bar{1}, 5, 7,  8,  9] \stackrel[4]{t_{\bar{1}4}}{\rightarrow} [4, \bar{3}, \bar{2}, 6,  \bar{1}, 5, 7,  8,  9].$$
We again apply again Step 1 to push $6$ and then $4$ forward to their natural positions:
\begin{eqnarray*}
 [4, \bar{3}, \bar{2}, {\bf 6},  \bar{1}, 5, 7,  8,  9]  &\stackrel[1]{t_{45}}{\rightarrow}&  [4, \bar{3}, \bar{2}, \bar{1}, {\bf 6}, 5, 7,  8,  9] \stackrel[1]{t_{56}}{\rightarrow}
[{\bf{4}}, \bar{3}, \bar{2}, \bar{1}, 5, 6, 7,  8,  9]\\
\stackrel[1]{t_{12}}{\rightarrow} [\bar{3}, {\bf{4}}, \bar{2}, \bar{1}, 5, 6, 7 , 8,  9]  &\stackrel[1]{t_{23}}{\rightarrow}& [\bar{3}, \bar{2}, {\bf{4}},  \bar{1}, 5, 6, 7,  8,  9] \stackrel[1]{t_{34}}{\rightarrow} [\bar{3}, \bar{2}, \bar{1}, 4, 5, 6, 7, 8, 9].
\end{eqnarray*}
We now unsign the pair $\bar{3}$ and $\bar{2}$:
$$[{\bf \bar{3}}, {\bf \bar{2}}, \bar{1}, 4, 5, 6, 7, 8, 9] \stackrel[2]{t_{\bar{1}2}}{\rightarrow}  [2, 3, \bar{1},  4, 5, 6, 7, 8, 9].$$

\noindent Now we again apply Step 1 to move $3$ and $2$ to their natural positions:
$$[2, {\bf 3}, \bar{1},  4, 5, 6, 7,  8,  9]
\stackrel[1]{t_{23}}{\rightarrow} [{\bf 2}, \bar{1}, 3, 4, 5, 6, 7, 8, 9]
\stackrel[1]{t_{12}}{\rightarrow} [\bar{1}, 2, 3, 4, 5, 6, 7, 8, 9].$$

\noindent Finally we unsign $1$:
$$[{\bf {\bar{1}}}, 2, 3, 4, 5, 6, 7, 8, 9] \stackrel[1]{t_{\bar{1}1}}{\rightarrow} [1, 2, 3, 4, 5, 6, 7, 8, 9],$$
and we are done. We obtained $w=t_{\bar{1}1}t_{12}t_{23}t_{\bar{1}2}t_{34}t_{23}t_{12}t_{56}
t_{45}t_{\bar{1}4}t_{57}t_{78}t_{47}t_{89}t_{78}$. The sum of the depths of the reflections in the decomposition is $22$, and this agrees with $d(w)=(8-4)+(7-5)+(9-7) - (-6-3-2-4-1)+\frac{1-5}{2}$, since $w$ is indecomposable.
\medskip

Note that, in $w$, $9$ is two places away from its natural position, so $9-w^{-1}(9)=2$. This is the cost we pay for moving $9$ to its place. Likewise, $8$ and $7$ contribute $4$ and $2$ respectively. The treatment of the pair $6$ and $4$, starting with their unsigning and ending with their arrival at their natural positions, costs $6+4-1=9$. The treatment of $2$ and $3$ costs $2+3-1$, and the unsigning of $1$ costs $1$.  Each of these costs corresponds to a specific contribution to $\dep(w)$ in Equation~\ref{eq:d_B}, and we will show this correspondence holds in general.
\end{exa}

As is clear from the description, given any $w \in B_n$, the algorithm retrieves the identity permutation. Now we show that the total depths of the transpositions applied in the algorithm is exactly $d(w)$. This will prove that $\dep(w)\leq d(w)$. The proof is based on the following four lemmas.

\begin{lem}\label{lem:typeBstep1cost}
Let $w \in B_n$ be indecomposable. Then the total cost of the first application of Step $1$ is $\sum_{w(i)>i}{(w(i)-i)}$, and the resulting permutation $u\in B_n$ has $d(u)=d(w)-\sum_{w(i)>i}{(w(i)-i)}$.
\end{lem}
\begin{proof}

Define $E(w):=\{w(i) \mid w(i)>i\}$. If $w(m)$ is the largest entry in $E(w)$, then the algorithm starts with several applications of Step $1$ to move $w(m)$ to its natural position.  Furthermore, no positive entry is moved to the left of its natural position in these moves since there will always be an entry $w(j)\leq m$ among the entries in positions $j$ with $m<j\leq w(m)$.  In fact, by our choice of $m$, $w(j)<w(m)$ for all such $j$, and, among the $w(m)-m$ distinct integers $\{w(m+1),\ldots,w(w(m))\}$, all of which are smaller than $w(m)$, one must be less than or equal to $m$.  Hence these moves bringing $w(m)$ to its natural position cost $w(m)-m$ and decrease 
$\sum_{w(i)>i}{(w(i)-i)}$ by $w(m)-m$.  Thus the first statement follows by induction on $\#E(w)$.

The second statement follows since there are no $i$ with $u(i)>i$, and Step 1 does not change any negative entries of $w$ (though it may move them around) or affect its oddness.
\end{proof}

\begin{obs}
Let $w\in B_n$, and assume that Step $1$ has been applied on $w$ until it cannot be continued, obtaining $u \in B_n$. Let $k$ denote the largest absolute value of a negative entry in $u$ (which is the same as in $w$). Then 

\begin{itemize}
\item The signed permutation $u$ consists of two parts: the leftmost $k$ entries, which form an indecomposable permutation $u'$, and the last $n-k$ entries, each of which is positive and in its natural position. 

\item All the positive entries in $u$ are located to the right of (or in) their natural positions.  In particular, every entry in $u$  of absolute value greater than $k$ is positive and in its natural position. 
\end{itemize}
\end{obs}

\begin{lem}
\label{lem:Bshuffleresult}Let $w \in B_n$ be indecomposable. Then Step 1 of the algorithm yields a permutation $u$ satisfying $u(i)<0$, $u(j)<0$, $|u(i)| \geq j$, and $|u(j)| \geq i$, where $|u(i)|$ and $|u(j)|$ are the two largest absolute values of negative entries and $i<j$.
\end{lem}

\begin{proof}

Let $k$ be defined as in the observation, and let $x:=|u(i)|$ and $y:=|u(j)|$. For any $p \leq k$, there must exist $q\leq p$ such that $u(q)<-p$.  (Indeed, there must be some $q\leq p$ with $|u(q)|>p$ since $u'$ is indecomposable, and we cannot have $u(q)>0$ since every positive entry is at or to the right of its natural position after Step $1$.) Applying this fact to $p=j-1$ immediately yields that $x \geq j$ since $-x$ is the smallest entry among the first $j-1$ entries.

Now, we show $y \geq i$. If $y>x$, then by what we have just shown we have $y>x\geq j>i$, so we are done. Otherwise, $y<x$. Consider the positions of the $x-y-1$ entries of absolute values $y+1,\ldots,x-1$. These entries must be positive since $x$ and $y$ are the two negative entries of largest absolute value. Hence after Step $1$, they must be at or to the right of their natural positions, and, in particular, to the right of position $y$.  Moreover, these entries must be at or to the left of position $x$ since $x$ is the smallest entry in $u$, so $k=x$ and $u(p)>x$ for all $p>x$.  If $y<i$, this means both $x$ and $y$ are in positions to the right of position $y$, and hence the entries $-y,y+1,\ldots,x-1,-x$ are in positions $y+1,\ldots, x=k$, a contradiction.  Therefore, we must have $y\geq i$.
\end{proof}

\begin{lem}
\label{lem:Bdoubleunsigncost}
Let $u \in B_n$ such that $u(m)<m$ for all $m$, and let $i<j$ be such that  $u(i)<0, u(j)<0$, $|w(i)| \geq j$ and $|w(j)| \geq i$. Then right multiplying $u$ by $t_{\bar{ij}}$ and repeatedly applying Step 1 of the algorithm until it can no longer be applied unsigns $w(i)$ and $w(j)$ and puts them in their natural positions.  These reflections cost in total exactly $|w(i)|+|w(j)|-1$.  The resulting permutation $u'\in B_n$ has $d(u')=d(u)-(|w(i)|+|w(j)|-1)$.
\end{lem}

\begin{proof}
Right multiplication by $t_{\bar i j}$ unsigns $w(i)$ and $w(j)$, switches their places, and costs $i+j-1$.  Subsequently, moving $|w(i)|$ to its place adds a further $|w(i)|-j$, for a total of $|w(i)|+i-1$.  By the same reasoning, moving $|w(j)|$ to its place adds $|w(j)|-i$ to the depth, yielding a total of $(|w(i)|+i-1)+(|w(j)|-i)=|w(i)|+|w(j)|-1$.  Furthermore, as in the proof of Lemma~\ref{lem:typeBstep1cost}, no positive entry is moved to the left of its natural position in the Step 1 moves.

For the second claim, note that there are no $m$ with $u(m)>m$ or $u'(m)>m$, the only change to negative entries is that $w(i)$ and $w(j)$ have been made positive, and $o^B(u)=o^B(u')$.
\end{proof}

Finally, we have

\begin{lem}
\label{lem:Bsingleunsigncost}
Let $u\in B_n$ such that $u(m)\leq m$ for all $m$, and suppose $\nneg(u)=1$.  Then $d(u)=|w(i)|$, where $w(i)$ is the unique negative entry.  Furthermore, applying Step 2 and then repeatedly applying Step 1 until it can no longer be applied costs exactly $|w(i)|$ and results in the identity signed permutation.
\end{lem}

\begin{proof}
First, note that we have $d(u)=|w(i)|+\frac{1-1}{2}$.  Now note that, since $\nneg(u)=1$, we must have $i=1$, meaning that the single negative entry must be in the leftmost position, since every positive entry is at or to the right of its natural position.  Therefore, unsigning $|w(i)|$ and moving it to its natural position add $1+(|w(i)|-1)=|w(i)|$ to $d(w)$ since, as in the proof of Lemma~\ref{lem:typeBstep1cost}, no positive entry is moved to the left of its natural position in the Step 1 moves.  These moves produce a signed permutation with no negative entries where every entry is at or to the right of its natural position, which must be the identity.
\end{proof}

\medskip

By the lemmas above, every time we run through a full series of Step 1 moves, or a Step 2 move followed by a series of Step 1 moves, we reduce $d(w)$ by exactly the cost of the moves we use.  Hence, by induction, we have shown that we can sort $w$ to the identity with cost $d(w)$.  This shows that $d(w)\geq \dep(w)$ for an indecomposable permutation $w$.  Since the algorithm is separately applied to each block of a decomposable permutation, and the formula $d(w)$ is additive over blocks, we can conclude the following.

\begin{prop}\label{prop:upperboundB}
For any $w\in B_n$, $d(w)\geq\dep(w)$.
\end{prop}

\subsection{Proofs of Theorem~\ref{thm:depthB} and subsequent corollaries.}

To prove that $d(w)\leq\dep(w)$ it is sufficient to prove the following lemma.

\begin{lem}\label{l:DeltaB} For any element $w \in B_n$ and any reflection $t\in T^B$,
$$d(w)-\dep(t) \leq d(wt).$$
\end{lem}

We now prove Theorem~\ref{thm:depthB} assuming this lemma.

\begin{proof}[Proof of Theorem~\ref{thm:depthB}]
By Proposition~\ref{prop:upperboundB}, $d(w)\geq\dep(w)$.  We now prove $d(w)\leq\dep(w)$ by induction on $\dep(w)$.  If $\dep(w)=0$, then $w=e$, and $d(e)=0$.  Otherwise, there exists a reflection $t\in T^B$ such that $\dep(w)-\dep(t)=\dep(wt)$.  By the inductive hypothesis, $\dep(wt)=d(wt)$.  Now, by Lemma~\ref{l:DeltaB}, $d(w)-\dep(t)\leq d(wt)=\dep(wt)=\dep(w)-d(t)$.  Hence $d(w)\leq\dep(w)$ as desired.
\end{proof}

Now we prove the lemma in cases for each type of reflection $t$.

\begin{proof}[Proof of Lemma \ref{l:DeltaB}]
Let us denote the three terms in Equation \eqref{eq:d_B} by $A, B$, and $C$ respectively. It will be convenient to let $\Delta d:=d(w)-d(wt)$ and $\Delta A, \Delta B$, and $\Delta C$ be the analogous differences. The claim can be proved by analyzing each type of reflection and the entries in the positions the reflection acts on. 

\noindent {\bf Case 1:}  $t=t_{ij}$, ($\dep(t)=j-i$). Clearly $\Delta B=0$, and $\nneg(w)=\nneg(wt)$.

\noindent a) $w(i)<i$.

The most change that $t$ can do to the block structure of $w$ is to join $i$ and $j$ into a single block, with every entry in between being a different singleton block.  This turns $j-i+1$ blocks into a single block.  Potentially, all of these blocks can be odd, and they are all fused into a single even block, so $o^B(w)$ can decrease by at most $j-i+1$.  Since no other terms of $d$ decrease, $d(w)-d(wt)\leq \frac{j-i+1}{2}\leq j-i$ since $j-i\geq 1$.  (Since we could have $w(j)>i$, the first term can change, but only to increase $d$.)
\smallskip

\noindent b) $i \leq w(i) <j$.

Note that $w(i)$ and $i$ are in the same block, so the most change $t$ can do to the block structure is to join $w(i)$ and $j$ into a single block, which decreases $o^B(w)$ by at most $j-w(i)+1$.  Furthermore, $$\Delta A =  w(i)-i + \left \{\begin{array}{ll} (w(j)-j)-(w(j)-i), & {\rm if} \ w(j)\geq j \\
- {\rm max} \{0,w(j)-i\}, & {\rm if} \ w(j)<j \end{array}.\right.$$ In both cases, $\Delta A \leq w(i)-i$.  Hence $\Delta d \leq w(i)-i+\frac{j-w(i)+1}{2}\leq j-i$ since $j-w(i)\geq 1$.
\smallskip

\noindent c)  $j\leq w(i)$.

In this case, blocks can only be split up and not merged, so $o^B(w)$ cannot decrease. Moreover $\Delta A=0$ unless $w(j)\leq j$, in which case $\Delta A =(w(i)-i)-(w(i)-j)-\max\{0,w(j)-i\} \leq j-i$.

\noindent {\bf Case 2:}  $t=t_{\bar{i}j}$, and $w(i)<0$ and $w(j)<0$ ($\dep(t)=i+j-1$). In this case $\Delta B= |w(i)|+|w(j)|$. Moreover, $\nneg(w)-\nneg(wt)=2$.

\noindent a) $|w(i)|<i$.

Suppose first that $|w(j)|\geq i$.  In this case $\Delta A=-(|w(j)|-i)$,  and $\Delta C\leq \frac{(j-i+1)-2}{2}.$ Hence
$\Delta d \leq i +|w(i)| +  \frac{j-i-1}{2}\leq \frac{3i+j-3}{2}\leq i+j-1$ since $|w(i)|\leq i-1$.

If $|w(j)|< i$, then $\Delta A=0$, and $\Delta C=-1$; hence $\Delta d=|w(i)|+|w(j)|-1\leq i+j-1.$
\smallskip

\noindent b) $i \leq |w(i)| <j$.

Suppose first that $|w(j)|\geq i$. In this case $\Delta A=-(|w(j)|-i)$,  and $\Delta C\leq \frac{(j-|w(i)|+1)-2}{2}.$ Hence
$\Delta d\leq i+|w(i)|+\frac{j-|w(i)|-1}{2}\leq \frac{2i+2j-2}{2}$ since $|w(i)| \leq j-1$.

If $|w(j)|< i$, then $A$ does not change, $\Delta C \leq -1$, and $\Delta d\leq i+j-1.$
\smallskip

\noindent c) $j\leq |w(i)|$.

Suppose first that $|w(j)| \geq i$.  In this case  $\Delta A \leq-(|w(i)|-j) - (|w(j)|-i)$, and $o^B(wt)\geq o^B(w)$, so $\Delta C\leq -1.$ Hence 
$\Delta d \leq i+j-1$. If  $|w(j)| < i$, then $\Delta A \leq-(|w(i)|-j)$ and $\Delta d \leq j +|w(j)|-1\leq i+j-1.$
\smallskip

\noindent {\bf Case 2':}  $t=t_{\bar{i}j}$, and $w(i)>0$ and $w(j)>0$.

\noindent a) $w(i)<i$.

In this case $\Delta A$ and $\Delta B$ are negative, and $\Delta C\leq \frac{j-i+3}{2}$.  Hence $\Delta d \leq i+j-1$.
\smallskip

\noindent b) $i \leq w(i) <j$.

In this case $\Delta A \leq (w(i)-i)+{\rm max}\{0, (w(j)-j)\}$, $\Delta B= -w(i)-w(j)$, and $\Delta C\leq \frac{j-w(i)+3}{2}$, and clearly
$\Delta d \leq i+ j-1$.

\noindent c) $j\leq w(i)$.

In this case $\Delta d$ is negative.

\noindent {\bf Case 2'':}  $t=t_{\bar{i}j}$, and $w(i)<0$ and $w(j)>0$.

\noindent a) and b) $|w(i)|<i$ or $i \leq |w(i)| <j$

In this case $\Delta A ={\rm max} \{0,w(j)-j\}$, $\Delta B=|w(i)|-w(j)$, and $\Delta C\leq \frac{j-i+1}{2}$. Hence $\Delta d \leq i+j-1$.
\smallskip

\noindent c) $j\leq w(i)$.

In this case, if $w(j)\geq j$, then  $\Delta A=w(j)-|w(i)|$,  $\Delta B=|w(i)|-w(j)$, and $\Delta C\leq \frac{j-i+1}{2}$. Hence $\Delta d \leq i+j-1$.

The symmetric case $w(i)>0$ and $w(j)<0$ is similar.

\noindent {\bf Case 3:}  $t=t_{\bar{i}i}$ and $w(i)<0$, ($\dep(t)=i$). In this case $\Delta B=|w(i)|$, and the block structure does not change.

\noindent a) $|w(i)|<i$

Clearly $A$ does not change, and $\Delta d = |w(i)| -1/2 \leq i-1/2 < i.$

\noindent b) $|w(i)|\geq i$

In this case $\Delta d < i$ since $\Delta A=-|w(i)|$ and $i\geq 1$.

\noindent {\bf Case 3'':}  $t=t_{\bar{i}i}$ and $w(i)>0$.  In this case $\Delta B=-w(i)$, and the block structure does not change.

\noindent a) $w(i)<i$

Clearly  $\Delta d$ is negative

\noindent b) $w(i)\geq i$

In this case $\Delta d = w(i)-i -w(i) +1/2 \leq i$, which is still negative.
\end{proof}

Now we prove Corollary~\ref{cor:depthBalt}.

\begin{proof}[Proof of Corollary~\ref{cor:depthBalt}]
Let $w\in B_n$.  First note that
$$\sumlim_{i=1}^n (i-|w(i)|)=0,$$
or equivalently
$$\sumlim_{w(i)>0} (i-w(i)) + \sumlim_{w(i)<0} (i+w(i)) = 0.$$
Therefore,
$$\sumlim_{w(i)>0} (i-w(i)) + \sumlim_{w(i)<0} (i-w(i)) = -2\sumlim_{w(i)<0} w(i) = 2\sumlim_{w(i)<0} |w(i)|,$$
and
$$\frac{1}{2}\sumlim_{i=1}^n (i-w(i))=\sumlim_{w(i)<0} |w(i)|.$$

Now add $\sumlim_{i>0, w(i)>i} (w(i)-i))$ to both sides.  We have
$$\frac{1}{2}\left(\sumlim_{i>0, w(i)>i} (w(i)-i)+\sumlim_{i>0, w(i)<i} (i-w(i))\right) = \sumlim_{i>0, w(i)>i} (w(i)-i) + \sumlim_{i>0, w(i)<0} |w(i)|,$$
or
$$\frac{1}{2}\left(\sumlim_{i>0, w(i)>i} (w(i)-i)+\sumlim_{i<0, w(i)>i} (w(i)-i)\right) = \sumlim_{i>0, w(i)>i} (w(i)-i) + \sumlim_{i>0, w(i)<0} |w(i)|,$$
which shows the equality of the formula in Corollary~\ref{cor:depthBalt} with that of Theorem~\ref{thm:depthB}.
\end{proof}

Now we prove Corollary~\ref{cor:maxdepthB}.

\begin{proof}[Proof of Corollary~\ref{cor:maxdepthB}]
Let $i \in [n]$ and consider the contribution of $w(i)$ to $\dep(w)$ in Theorem~\ref{thm:depthB}. If $w(i)>0$, then $w(i)$ contributes at most $w(i)-1$ to the first summand. On the other hand, if $w(i)<0$, then it always contributes $|w(i)|$ to the second summand and $-\frac{1}{2}$ to the third, for a total contribution of $|w(i)|-\frac{1}{2}$.  In either case, $w(i)$ contributes at most an additional $\frac{1}{2}$ to the third summand (if $w(i)<0$ is a block by itself).  Hence, the greatest contribution is made if all entries are negative and each constitutes a full block. 
\end{proof}

\subsection{Reducedness of factorization}

Let $w \in B_n$ be indecomposable. Use the algorithm to write $w$ as a product of reflections $w=t_1 \cdots t_r$ realizing the depth of $w$.  Then, for each $k \in [r]$, depending on whether $t_k$ is equal to  $t_{ij}$, $t_{\bar{i} i}$, or $t_{\bar{i} j}$, replace $t_k$ by the following reduced expression:
\begin{eqnarray*}
t_{ij}& {\rm by} & s_{j-1} \cdots s_{i+1} s_{i} s_{i+1} \cdots s_{j-1}, \\
t_{\bar{i} i}& {\rm by} & s_{i-1}\cdots s_1  s_0^B s_1 \cdots s_{i-1}, \ {\rm and} \\
t_{\bar{i} j}& {\rm by} & s_{i-1}\cdots s_1  s_{j-1} \cdots s_2 s_0^B s_1 s_0^B s_2 \cdots s_{j-1}s_1 s_2 \cdots s_{i-1}.
\end{eqnarray*}

To prove Theorem~\ref{thm:typeBred}, we show the following.

\begin{thm}\label{reduced}
The decomposition of $w=t_1\cdots t_r$ given by the algorithm, where every $t_i$ is replaced as explained above, is reduced. In particular our algorithm defines a chain from the identity to $w$ in the right weak order of $W$.
\end{thm}

\begin{proof}
It is sufficient to show that in each step of the algorithm 
$$\ell_S(wt)=\ell_S(w)-\ell_S(t).$$
We discuss the following three cases.

\noindent {\bf Step 1:} The reflections applied in Step 1 are of the form $t=t_{ij}$, where $w(i)$ is the maximal positive entry in $w$ and $w(j)$ is the minimal entry among those with $i<j<w(i)$.  Hence, for any $i<k<j$, one has $w(i)>w(k)>w(j)$.  Therefore, the result follows by Lemma~\ref{lem:typeBreflprodred}(1).

\noindent {\bf Step 2, Case A:} In Step 2, we either apply $t=t_{\bar{i}j}$ or $t_{\bar{i}i}$.  If we apply $t=t_{\bar{i}j}$, then, for any $k\in[j-1]\setminus\{i\}$, we have $|w(k)|<{\rm min} \{|w(i)|,|w(j)|\}$. Hence the result follows by Lemma~\ref{lem:typeBreflprodred}(2).

\noindent {\bf Step 2, Case B:} If we apply $t=t_{\bar{i}i}$ in Step 2, then $w(i)$ is the first entry in a indecomposable block $w^p$ of $w$, as argued in the proof of Lemma~\ref{lem:Bsingleunsigncost}, and $w(i)<0$. All other entries in $w^p$ are positive, and all entries in the previous blocks of $w$ are smaller in absolute value then $|w(i)|$. Hence the result follows from Lemma~\ref{lem:typeBreflprodred}(3).
\end{proof}

\section{Realizing depth in type D}\label{sec:typeDalg}

\subsection{Reflections, length, and depth in type D}  For the reader's convenience we state here the basic facts on reflections, length, and depth for the Coxeter system $(D_n, S_D)$.  The facts about reflections and length can be found in~\cite[\S 8.2]{BBbook}.

In the group $D_n$ the set of reflections is given by
$$T^D:=\{t_{ij}, t_{\bar{i}j}   \mid 1 \leq i < j \leq n\},$$
where $t_{ij}=(i,j)(\bar{i},\bar{j})$ and  $t_{\bar{i}j}=(\bar{i},j)(i,\bar{j})$ in cycle notation. In particular there are $n^2-n$ reflections in $D_n$.  These reflections act on signed permutations in window notation just as they do in type B.

The length of a signed permutation $w\in D_n$ is given by
\begin{equation}\label{eq:typeDlength}
\ell_S(w)=\inv(w)+\nsp(w),
\end{equation}
where $\inv(w)$ and $\nsp(w)$ are defined as in type B.  This means the only difference is that $\nneg(w)$ contributes to length in type B but not in type D.  As in type B, a pair $(i,j)$ may appear in both $\Inv(w)$ and $\Nsp(w)$, in which case this pair is counted twice in calculating $\ell_S(w)$.

From Equations (\ref{eq:depthreflection}) and (\ref{eq:typeDlength}) we immediately obtain the depths of the two types of reflections

\begin{lem}\label{lem:typeDdepthreflections}
The depths of reflections in type D are as follows.
$$\dep(t_{ij})=j-i \quad and \quad \dep(t_{\bar{i}j})=i+j-2.$$
\end{lem}

The rules for determining if $w=vt$ is a reduced factorization are the same as in type B, with one minor change.

\begin{lem}\label{lem:typeDreflprodred}
Let $v, w\in D_n$, $t\in T^D$, and $w=vt$.  Then $\ell_S(w)=\ell_S(v)+\ell_S(t)$ if and only if one of the following hold:
\begin{enumerate}
\item $t=t_{ij}$, $w(i)>w(j)$, and for all $k$ with $i<k<j$, $w(i)>w(k)>w(j)$.
\item $t=t_{\bar{i}j}$, $w(i)+w(j)<0$, for all $k$ with $1\leq k<i$, $w(k)>w(i)$ and $w(k)+w(j)<0$, and for all $k'\neq i$ with $1\leq k'<j$ , $w(i)+w(k')\leq 0$ and $w(k')>w(j)$.
\end{enumerate}
\end{lem}

\begin{proof}

As in the proof of Lemma~\ref{lem:typeBreflprodred}, we write each $t$ as a product of simple reflections $t=s_{i_1} \cdots s_{i_r}$ and show that each of the conditions of the lemma corresponds to the assertion that for each $k \in [r]$ one has $\ell_S(ws_{i_1} \cdots s_{i_{k-1}} s_{i_k})<\ell_S(ws_{i_1} \cdots s_{i_{k-1}})$. 

\begin{enumerate}
\item If $t=t_{ij}$, then the proof is exactly as for the analogous case in Lemma~\ref{lem:typeBreflprodred}.

\item If $t=t_{\bar{i}j}$ then its decomposition into simple reflections is $$t=s_{i-1}\cdots s_1  s_{j-1} \cdots s_2 s_0^D s_2 \cdots s_{j-1}s_1 s_2 \cdots s_{i-1}.$$ 

The proof here is the same as for Lemma~\ref{lem:typeBreflprodred}, except the central $s_0^B s_1 s_0^B$ has become an $s_0^D$, which causes the conditions $w(i),w(j)<0$ to be replaced with the condition $w(i)+w(j)<0$.
\end{enumerate}
\end{proof}

\subsection{Algorithm for realizing depth in type D}

As before, we let $d(w)$ denote our expected value of $\dep(w)$ according to our formula.  Hence, for $w\in D_n$, let
\begin{equation*}
d(w):=\sum\limits_{\{i \in [n] \mid w(i)>i\}}{(w(i)-i)} + \sumlim_{i \in \Neg(w)}{|w(i)|}+o^D(w)-\nneg(w).
\end{equation*}

In order to prove Theorem \ref{thm:depthD}, we proceed in two steps. First, we supply an algorithm that associates to each $w \in D_n$  a decomposition of $w$ into a product of reflections whose sum of depths is $d(w)$. This will prove that $d(w)\geq\dep(w)$. Then we will show that $d(w)\leq\dep(w)$.

As in type B, we will present our algorithm as an algorithm to sort $w$ to the identity signed permutation $[1,\ldots,n]$ by a sequence of right multiplications by reflections $t_{i_k j_k}$.  One can then recover a decomposition of $w$ as the product of these reflections in reverse order.  Our algorithm is applied on each type D block of $w$ separately; hence from now on we assume that $w$ is type D indecomposable.

\begin{alg} Let $w \in D_n$ be type D indecomposable and $w=w_1\oplus\cdots \oplus w_b$ its type B decomposition.
\begin{enumerate}
\item If every positive entry in $w_b$ is to the right of its natural position, then proceed to the next step.  Otherwise, let $i$ be the position such that $w(i)$ is maximal among all entries (necessarily in $w_b$) with $w(i)>i$.  Now let $j$ be the position such that $w(j)$ is minimal among all $w(j)$ for which $i<j \leq w(i)$.  Right multiply by $t_{ij}$.  Repeat this step until $w(i)\leq i$ for all $i$ in $w_b$.

\item If our signed permutation is now type B indecomposable (disregarding any positive entries with absolute value greater than that of any negative entry, which are now all in their natural positions), proceed to the next step.  Otherwise, let $i$ be the index of the first position in the last type B block, and right multiply by $t_{(i-1)i}$.  Then go back to Step 1.

\item Right multiply by $t_{\bar{i}j}$, where $w(i)$ and $w(j)$ are the two negative entries of largest absolute value, and go back to Step 1.
\end{enumerate}
\end{alg}

The algorithm begins by shuffling each positive entry $w(i)$ in the rightmost B-block that appears to the left of its natural position into its natural position, starting from the largest and continuing in descending order.  Once this is completed, we join the two rightmost B-blocks (not counting the string of positive entries in their natural positions at the far right) by a simple reflection and continue moving positive entries that are to the left of their natural positions in the newly enlarged rightmost B-block.  At the end of this process, our signed permutation (which was assumed to be D-indecomposable) is B-indecomposable (excepting the string of positive entries in their natural positions at the far right), and every positive entry is to the right of its natural position.  Then we unsign a pair, thus obtaining two new positive entries.  The entire process restarts, and the remaining entries might be further shuffled.  Unsigning and shuffling moves continue to alternate until neither type of move can be performed.

At the end of the algorithm, there are no negative entries, and no positive entry is to the left of its natural position.  Hence we must have the identity signed permutation.

As in the type B algorithm, it can happen that $w$ is transformed into a D-decomposable permutation, but this does not pose a problem since the only way this occurs is by creating blocks on the right consisting entirely of positive entries that will be further acted upon by Step 1 to move them to their natural positions, and indecomposability only matters in ensuring we do not join B-blocks from different D-blocks in Step 2.

We give an example of our algorithm.

\begin{exa}

Let $w=[5, \bar1, \bar3, 2, \bar4, 6, 9, \bar8, 7]\in D_9$.  Note $w$ is D-indecomposable but has 3 B-blocks.  We have $d(w)=(5-1)+(9-7)+(1+3+4+8)+2-4=20$.  Our first step is to shuffle entry 9 to position 9:
$$w= [5, \bar1, \bar3, 2, \bar4, 6, {\bf 9}, \bar8, 7]
\stackrel[1]{t_{78}}{\rightarrow}
[5, \bar1, \bar3, 2, \bar4, 6, \bar8, {\bf 9}, 7]
\stackrel[1]{t_{89}}{\rightarrow}
[5, \bar1, \bar3, 2, \bar4, 6, \bar8, 7, 9].$$

Now we apply Step 2:
$$[5, \bar1, \bar3, 2, \bar4, {\bf 6}, {\bf \bar8}, 7, 9]
\stackrel[1]{t_{67}}{\rightarrow}
[5, \bar1, \bar3, 2, \bar4, \bar8, 6,7, 9].$$

We have no positive entries in the rightmost relevant B-block (ignoring the rightmost block consisting of the 9 by itself) that are to the left of their natural positions, so we bypass Step 1 and apply Step 2 again:
$$[5, \bar1, \bar3, 2, {\bf \bar4}, {\bf \bar8}, 6, 7, 9]
\stackrel[1]{t_{56}}{\rightarrow}
[5, \bar1, \bar3, 2, \bar8, \bar4, 6,7, 9].$$

Now we apply Step 1:
$$[{\bf 5}, \bar1, \bar3, 2, \bar8, \bar4, 6, 7, 9]
\stackrel[4]{t_{15}}{\rightarrow}
[\bar8, \bar1, \bar3, 2, 5, \bar4, 6,7, 9].$$

Then we unsign the two most negative entries:
$$[{\bf \bar8}, \bar1, \bar3, 2, 5, {\bf \bar4}, 6, 7, 9]
\stackrel[5]{t_{\bar1 6}}{\rightarrow}
[4, \bar1, \bar3, 2, 5, 8, 6,7, 9].$$

We then push 8 and then 4 forward to their natural positions:
\begin{eqnarray*}
[4, \bar1, \bar3, 2, 5, {\bf 8}, 6,7, 9] &\stackrel[1]{t_{67}}{\rightarrow}& [4, \bar1, \bar3, 2, 5, 6,{\bf 8},7, 9] \stackrel[1]{t_{78}}{\rightarrow}
[{\bf 4}, \bar1, \bar3, 2, 5, 6,7, 8, 9]\\
&\stackrel[2]{t_{13}}{\rightarrow}& [\bar3, \bar1, {\bf 4}, 2, 5, 6, 7, 8, 9]  \stackrel[1]{t_{34}}{\rightarrow}[\bar3, \bar1, 2, 4, 5, 6, 7, 8, 9].
\end{eqnarray*}

Finally, we unsign the 1 and the 3 and push the 3 forward to its natural position:
$$[{\bf \bar3}, {\bf \bar1}, 2, 4, 5, 6, 7, 8, 9]
\stackrel[1]{t_{\bar12}}{\rightarrow}
[1,{\bf 3},2,4,5,6,7,8,9]
\stackrel[1]{t_{23}}{\rightarrow}
[1,2,3,4,5,6,7,8,9],$$
and we are done.  The decomposition $w=t_{23}t_{\bar12}t_{34}t_{13}t_{78}t_{67}t_{\bar16}t_{15}t_{56}t_{67}t_{89}t_{78}$ is obtained, and the sum of depths of the reflections involved is 20, as expected.
\end{exa}

Note that $9-7=2$ is the cost we pay to move 9 to its place, and $5-1=4$ is the cost we pay to move 5 to its place.  Furthermore, we pay $2=o^D(w)$ in Step 2 moves to join $w$ into a single B-block.  The treatment of the pair 8 and 4, from their unsigning to their arrival at their natural positions, costs $8+4-2=10$, and the treatment of the 1 and 3 together costs $1+3-2=2$.

Now we show, as suggested, that the total depths of the reflections applied in the algorithm is exactly $d(w)$.  This will prove that $\dep(w)\leq d(w)$.  The proof is based on the following lemmas.  We will assume that $w$ is D-indecomposable with $b$ B-blocks, and, for $k\in[b]$, we let $a_k$ be the index of the last entry in the $k$-th B-block, so that the $k$-th B-block has $a_k-a_{k-1}$ entries.

\begin{lem}
Let $w\in D_n$ be D-indecomposable.  Then the total cost of the first series of applications of Step 1 before the first application of Step 2 or Step 3 is $\sum_{i>a_{b-1}, w(i)>i} (w(i)-i)$.
\end{lem}

The proof for this lemma is entirely identical to that of Lemma~\ref{lem:typeBstep1cost}.

\begin{lem}
\label{lem:Dstep2conditions}
Suppose $w$ is D-indecomposable, and let $u$ be the result of a series of applications of Step 1 after $c-1$ applications of Step 2 to $w$.  Unless $u$ is the identity, $u(a_{b-c}+1)<0$.  Furthermore, the cost of the $k$-th application of Step 2 and the ensuing series of applications of Step 1 is $1+\sum_{a_{b-c-1}<i\leq a_{b-c}, u(i)>i} (u(i)-i) = 1+\sum_{a_{b-c-1}<i\leq a_{b-c}, w(i)>i} (w(i)-i)$.  Hence, the total cost of all applications of Step 1 and Step 2 before the first application of Step 3 is $\sum_{w(i)>i} (w(i)-i) + o^D(w)$.

Furthermore, if $v$ is the result of all applications of Step 1 and Step 2 before the first application of Step 3, then $v$ consists of one B-indecomposable (and D-indecomposable) block in positions $[1,k]$ and positive entries all in their natural positions further to the right, where $k$ is the absolute value of the most negative entry of $w$.
\end{lem}

\begin{proof}
We proceed by induction on $c$.  In the base case where $c=1$, since $u$ is not the identity and $w$ is D-indecomposable, the rightmost B-block of $w$ must have a negative entry.  Assume for contradiction that $u(a_{b-1}+1)>0$, and let $j$ be the leftmost position in the $b-c+1$-th B-block such that $u(j)<0$.  Then we must have $0<w(i)\leq i$ for all $i$ with $a_{b-c+1}<i<j$ since, otherwise, $w(i)$ would have been swapped with $u(j)$ or some other negative entry when it was moved.  (It is possible that $u(j)$ started further to the right, but if the entry $u(j)$ was to the right of the position $w(i)$ when $w(i)$ was moved, then it is to the right of position $p$ for any $p<w(i)$, and hence it would never have been moved to the $j$-th position.)  This contradicts the assumption that $[a_{b-1}+1,n]$ forms a single B-block.

Now we treat the inductive case.  By the inductive hypothesis, when the reflection $t_{a_{b-c+1},a_{b-c+1}+1}$ in Step 2 is applied to a signed permutation $u$, $u(a_{b-c+1})\leq a_{b-c+1}$ (since $a_{b-c+1}$ is the rightmost position in a B-block), and $u(a_{b-c+1}+1)<0$.  Hence this reflection does not affect the set $\{i\in[a_{b-c}+1,a_{b-c+1}]\mid u(i)>i\}$, and, as no other reflections have previously been applied to these positions, $\{i\in[a_{b-c}+1,a_{b-c+1}]\mid u(i)>i\}=\{i\in[a_{b-c}+1,a_{b-c+1}]\mid w(i)>i\}$.  Hence by the argument of Lemma~\ref{lem:typeBstep1cost}, the total cost of the single Step 2 move and the ensuing series of Step 1 moves is $1+\sum_{a_{b-c}<i\leq a_{b-c+1}, w(i)>i} (w(i)-i)$.  Furthermore, by the same argument as in the previous paragraph, $u(a_{b-c}+1)<0$.

Since we started with $o^D(w)+1$ B-blocks, we performed $o^D(w)$ Step 2 moves, each of which cost depth 1, so the cost of all these moves is $\sumlim_{w(i)>i} (w(i)-i) + o^D(w)$.

If $v$ is the signed permutation produced by the $c$-th application of Step 2, we must have $v(a_{b-c}+1)\leq a_{b-c}$, and this entry is not moved before the first application of Step 3.  Hence, positions $a_{b-c}$ and $a_{b-c}+1$ will be in the same block.  Furthermore, since $v(a_{b-c})<0$ and $|v(a_{b-c})|$ is larger than the absolute value of any entry in the $b-c$-th B-block, the ensuing series of Step 1 moves will not break this B-block (and further Step 2 and Step 1 moves are all to its left), so $a_{b-c-1}+1$ and $a_{b-c}$ remain in the same B-block.  Hence the resulting permutation is B-indecomposable except possibly for a string of positive entries all in their natural positions at the far right.
\end{proof}

\begin{lem}
Let $w \in D_n$ be D-indecomposable. Then, whenever Step 3 is performed, we have a permutation $u$ satisfying $u(i)<0$, $u(j)<0$, $|u(i)| \geq j$, and $|u(j)| \geq i$, where $|u(i)|$ and $|u(j)|$ are the two largest values of negative entries and $i<j$.
\end{lem}

Since $u$ is B-indecomposable (except for the entries at the far right), the proof is the same as for Lemma~\ref{lem:Bshuffleresult}.

\begin{lem}
Let $u \in D_n$ be B-indecomposable such that $u(m)<m$ for all $m$, and let $i<j$ be such that  $u(i)<0, u(j)<0$, $|w(i)| \geq j$ and $|w(j)| \geq i$. Then right multiplying $u$ by $t_{\bar{i}j}$ and repeatedly applying Step 1 of the algorithm until it can no longer be applied will unsign $w(i)$ and $w(j)$ and put them in their natural positions.  These reflections cost in total exactly $|w(i)|+|w(j)|-2$.  The resulting permutation $u'\in D_n$ has $d(u')=d(u)-(|w(i)|+|w(j)|-2)$ and is B-indecomposable (except possibly for positive entries in their natural positions at the far right).
\end{lem}

The proof is entirely identical to the proof of Lemma~\ref{lem:Bdoubleunsigncost}, except that multiplying by $t_{\bar{i}j}$ costs $i+j-2$ rather than $i+j-1$.

By the lemmas above, the total cost of the algorithm applied to $w\in D_n$, a D-indecomposable signed permutation, is $d(w)$.  Since the algorithm is separately applied to each D-block of a D-decomposable permutation, and the formula $d(w)$ is additive over D-blocks, we can conclude the following.

\begin{prop}\label{prop:upperboundD}
For any $w\in D_n$, $d(w)\geq\dep(w)$.
\end{prop}

\subsection{Proof of Theorem~\ref{thm:depthD} and subsequent corollaries}

To prove that $d(w)\leq\dep(w)$ it is sufficient to prove the following lemma.

\begin{lem}\label{l:DeltaD} For any element $w \in D_n$ and any reflection $t\in T^D$,
$$d(w)-\dep(t) \leq d(wt).$$
\end{lem}

We now prove Theorem~\ref{thm:depthD} assuming this lemma.

\begin{proof}[Proof of Theorem~\ref{thm:depthD}]
By Proposition~\ref{prop:upperboundD}, $d(w)\geq\dep(w)$.  We now prove $d(w)\leq\dep(w)$ by induction on $\dep(w)$.  If $\dep(w)=0$, then $w=e$, and $d(e)=0$.  Otherwise, there exists a reflection $t\in T^D$ such that $\dep(w)-\dep(t)=\dep(wt)$.  By the inductive hypothesis, $\dep(wt)=d(wt)$.  Now, by Lemma~\ref{l:DeltaD}, $d(w)-\dep(t)\leq d(wt)=\dep(wt)=\dep(w)-d(t)$.  Hence $d(w)\leq\dep(w)$ as desired.
\end{proof}

Now we prove the lemma in cases for each type of reflection $t$.  The proof is identical to that of Lemma~\ref{l:DeltaB} except for differences caused by the different definitions of oddness and the different depths of $t_{\bar{i}j}$ in the two cases.  We will only note the differences and refer to the previous proof where possible.

\begin{proof}[Proof of Lemma \ref{l:DeltaD}]

As in the proof of Lemma~\ref{l:DeltaB}, we denote the three terms in Equation \eqref{eq:d_D} by $A, B$, and $C$ respectively and let $\Delta d:=d(w)-d(wt)$ and $\Delta A, \Delta B$, and $\Delta C$ be the analogous differences. Then we analyze each type of reflection and the entries in the positions the reflection acts on. 

\noindent {\bf Case 1:}  $t=t_{ij}$, ($\dep(t)=j-i$). Clearly $\Delta B=0$.

\noindent a) $w(i)<i$.
Again, the most change $t$ can do to the B-block structure of $w$ is to join $i$ and $j$ into a single B-block, with each entry in between being a different singleton B-block.  Then $t$ turns $j-i+1$ B-blocks into a single B-block.  In the worse case, $i$ and $j$ are in the same D-block, in which case $o^D(w)$ can decrease by $j-i$.  Since $\Delta A=0$ in this case, $\Delta d\leq j-i$.

\noindent b) $i \leq w(i) <j$.

Combining the analogous argument in Lemma~\ref{l:DeltaB} and Case 1(a) above, we see $\Delta A\leq w(i)-i$ and $\Delta C\leq j- w(i)$, so $\Delta d\leq j-i$.

\noindent c)  $j\leq w(i)$.

This case is exactly as in Lemma~\ref{l:DeltaB}.

\noindent {\bf Case 2:}  $t=t_{\bar{i}j}$, and $w(i)<0$ and $w(j)<0$ ($\dep(t)=i+j-2$). In this case $\Delta B= |w(i)|+|w(j)|-2$.

\noindent a) $|w(i)|<i$.

Suppose first that $|w(j)|\geq i$.  In this case $\Delta A=-(|w(j)|-i)$,  and $\Delta C\leq j-i.$ Hence
$\Delta d \leq |w(i)| +  j -2 \leq i+j-2$.

If $|w(j)|< i$, then $\Delta A=0$, and $\Delta C=0$; hence $\Delta d=|w(i)|+|w(j)|-2\leq i+j-2.$

\noindent b) $i \leq |w(i)| <j$.

Suppose first that $|w(j)|\geq i$. In this case $\Delta A=-(|w(j)|-i)$,  and $\Delta C\leq j-|w(i)|.$ Hence
$\Delta d\leq i+j-2$.

If $|w(j)|< i$, then $A$ does not change, $\Delta C \leq 0$, and $\Delta d\leq i+j-2.$

\noindent c) $j\leq |w(i)|$.

Suppose first that $|w(j)| \geq i$.  In this case  $\Delta A \leq-(|w(i)|-j) - (|w(j)|-i)$, and $o^D(wt)\geq o^D(w)$, so $\Delta C\leq 0.$ Hence 
$\Delta d \leq i+j-2$. If  $|w(j)| < i$, then $\Delta A \leq-(|w(i)|-j)$ and $\Delta d \leq j +|w(j)|-2\leq i+j-2.$
\smallskip

\noindent {\bf Case 2':}  $t=t_{\bar{i}j}$, and $w(i)>0$ and $w(j)>0$.

\noindent a) $w(i)<i$.

In this case $\Delta A$ and $\Delta B$ are nonpositive and $\Delta C\leq j-i$.  Hence $\Delta d \leq i+j-2$ as $i\geq 1$.
\smallskip

\noindent b) $i \leq w(i) <j$.

In this case $\Delta A \leq (w(i)-i)+{\rm max}\{0, (w(j)-j)\}$, $\Delta B= -w(i)-w(j)+2$, and $\Delta C\leq j-w(i)$, and clearly
$\Delta d \leq i+ j-2$.

\noindent c) $j\leq w(i)$.

In this case $\Delta d$ is negative.

\noindent {\bf Case 2'':}  $t=t_{\bar{i}j}$, and $w(i)<0$ and $w(j)>0$.

\noindent This is the same as the analogous case for Lemma~\ref{l:DeltaB}, except that $\Delta C\leq j-i$ instead, and we can conclude that $\Delta d\leq i+j-2$.

The symmetric case $w(i)>0$ and $w(j)<0$ is similar.

\end{proof}

The proof of Corollary~\ref{cor:depthDalt} from Theorem~\ref{thm:depthD} is exactly the same as the proof of Corollary~\ref{cor:depthBalt} from Theorem~\ref{thm:depthB}. Now we prove Corollary~\ref{cor:maxdepthD}.

\begin{proof}[Proof of Corollary~\ref{cor:maxdepthD}]

Let $i \in [n]$, and consider the contribution of $w(i)$ to $\dep(w)$ in Equation~\eqref{eq:d_D}. If $w(i)>0$, then $w(i)$ contributes at most $w(i)-i$ to the first summand. On the other hand, if $w(i)<0$, then it always contributes $|w(i)|-1$ to the second summand. Hence it is clear that a maximal element must have at most one positive entry. In fact, if $w$ had two positive entries $w(i),w(j)$, the D-oddness could be greater than $n/2$ but could never make up for the $w(i)+w(j)-2$ that would be lost from the second term of Equation~\eqref{eq:d_D}. 

If $n$ is even the elements with maximal D-oddness (equal to $n/2$) must be of the form 
$$ [\bar{1}, |\bar{2},\bar{3}|^{\circlearrowleft},\ldots |\overline{n-2},\overline{n-1}|^{\circlearrowleft},\bar{n}],$$
where the notation $|\bar{i},\overline{i+1}|^{\circlearrowleft}$ means that we can consider the entries $\bar{i}$ and $\overline{i+1}$ in either of the two possible orders.
The number of all such elements is clearly $2^{\frac{n-2}{2}}$ since for any of the $(n-2)/2$ pairs we have two choices for ordering its entries.
The depth of such elements is equal to ${n \choose 2}+n/2$. 

If $n$ is odd, then the set of maximal elements splits into four families of elements of the form
$$ [|*,*| |\bar{3},\bar{4}|^{\circlearrowleft},\ldots |\overline{n-2},\overline{n-1}|^{\circlearrowleft},\bar{n}],$$
where the pair $|*,*|$ can be $[1,\bar{2}]$, $[2,\bar{1}]$, $[\bar{2},1]$, or $[\bar{1},2]$.  The D-oddness is $(n-1)/2$ in the first three cases, and $(n+1)/2$ in the last one. In all cases we reach the maximum depth ${n \choose 2}+(n-1)/2$. Clearly there are $4\times 2^{\frac{n-3}{2}}$ such elements, and we are done.
\end{proof}

\subsection{Reducedness of factorization}

Let $w \in D_n$ be indecomposable. Use our algorithm to write $w$ as a product of reflections $w=t_1 \cdots t_r$ realizing the depth of $w$.  Then, for each $k \in [r]$, depending on whether $t_k$ is equal to  $t_{ij}$ or $t_{\bar{i} j}$, replace $t_k$ by the following product of simple reflections:
\begin{eqnarray*}
t_{ij} &{\rm by} & s_{j-1} \cdots s_{i+1} s_{i} s_{i+1} \cdots s_{j-1}, \ {\rm and} \\
t_{\bar{i} j} & {\rm by} & s_{i-1}\cdots s_1  s_{j-1} \cdots s_2 s_0^D  s_2 \cdots s_{j-1}s_1 s_2 \cdots s_{i-1}.
\end{eqnarray*}
To prove Theorem~\ref{thm:typeDred}, we show the following.

\begin{thm}\label{Dreduced}
The decomposition of $w=t_1\cdots t_r$ given by the algorithm, where every $t_i$ is replaced as explained above, is reduced. In particular our algorithm defines a chain from the identity to $w$ in the right weak order of $W$.
\end{thm}

\begin{proof}
It is sufficient to show that in each step of the algorithm 
$$\ell_S(wt)=\ell_S(w)-\ell_S(t).$$

We discuss the following three cases.

\noindent {\bf Step 1:} The reflections applied in Step 1 are of the form $t=t_{ij}$, where $w(i)$ is the maximal positive entry in $w$ and $w(j)$ is the minimal entry among those with $i<j<w(i)$.  Hence, for any $i<k<j$ one has $w(i)>w(k)>w(j)$.  Therefore, the result follows by Lemma~\ref{lem:typeDreflprodred}(1).

\noindent {\bf Step 2:} In Step 2, we apply $t_{(i-1)i}$, where $i$ is the index of the first position in the last type B block.  By Lemma~\ref{lem:Dstep2conditions}, $w(i)<0$, and since $w(i-1)$ and $w(i)$ are in different B-blocks, $|w(i-1)|<|w(i)|$.  Hence, the result follows by Lemma~\ref{lem:typeDreflprodred}(1).

\noindent {\bf Step 3:} We apply $t=t_{\bar{i}j}$, where, for any $k\in[j-1]\setminus\{i\}$, we have $|w(k)|<{\rm min} \{|w(i)|,|w(j)|\}$.  Hence the result follows by Lemma~\ref{lem:typeDreflprodred}(2).
\end{proof}

\section{Coincidences of length, depth, and reflection length}\label{sec:coincidences}

\subsection{Coincidence of length and depth}
\label{sec:three}

In this section we consider some of the consequences of Theorems~\ref{reduced} and~\ref{Dreduced}.  First, we characterize of the elements $w\in B_n$ and $w\in D_n$ satisfying $\dep(w)=\ell_S(w)$.  For the symmetric group, Petersen and Tenner showed that the corresponding elements are precisely the fully commutative ones~\cite[Thm. 4.1]{PT}.  The situation in $B_n$ and $D_n$ is similar.  Indeed, our proof covers the case of the symmetric group $S_n$ and is different from that of Petersen and Tenner.

Since our considerations apply to any Coxeter group, we begin with a general definition.  Let $(W,S)$ be a Coxeter system, and let $w\in W$.  We say that the depth of $w$ is {\bf realized by a reduced factorization} if there exist $t_1,\ldots, t_r\in T$ such that $w=t_1\cdots t_r$, $\dep(w)=\sum_{i=1}^r \dep(t_r)$, and $\ell_S(w)=\sum_{i=1}^r \ell_S(t_r)$.  By Theorems~\ref{reduced} and~\ref{Dreduced}, the depth of every element of $B_n$ and $D_n$ is realized by a reduced factorization.  Note that we can also apply our algorithm for $B_n$ (or $D_n$) to its subgroup of unsigned permutations to get an analogous result for $S_n$.  If the depth of every element of $(W,S)$ is realized by a reduced factorization, we say that depth is {\bf universally realized} by reduced factorizations in $(W,S)$.  We do not know of any examples where depth is not realized by a reduced factorization.

Following Fan~\cite{Fan}, we say that $w\in W$ is {\bf short-braid-avoiding} if there does not exist a consecutive subexpression $s_i s_j s_i$, where $s_i, s_j\in S$, in any reduced expression for $w$.  Note that, if $s_i s_j s_i$ appears in a reduced expression, then  $s_i$ and $s_j$ cannot commute.  As Fan remarks, for elements of a simply-laced Coxeter group, being short-braid-avoiding is equivalent to being fully commutative, and for non-simply-laced groups, the short-braid-avoiding elements form a subset of the fully commutative ones.

Our characterization of elements for which length and depth coincide is as follows.

\begin{thm}
Let $(W,S)$ be any Coxeter group.  For any $w\in W$, $\dep(w)=\ell_S(w)$ if and only if $w$ is short-braid-avoiding and the depth of $w$ is realized by a reduced factorization.  Hence, for a Coxeter group $(W,S)$ in which depth is universally realized by reduced factorizations, an element $w\in W$ satisfies the equality $\dep(w)=\ell_S(w)$ if and only if $w$ is short-braid-avoiding.
\end{thm}

\begin{proof}
Suppose that $w\in W$ is not short-braid-avoiding.  Then $w$ is a reduced product of the form $us_is_js_iv$ for some $u,v\in W$ and $s_i,s_j\in S$.  Hence $\dep(w)\leq \dep(u)+2+\dep(v)<\ell_S(w)$.

Now suppose the depth of $w$ is not realized by a reduced factorization (using the set of reflections $T$).  Then, in particular, the depth of $w$ is not realized by a reduced expression (using only the set of simple reflections $S$).  Hence $\dep(w)<\ell_S(w)$.

Note that $\dep(w)\leq \ell_S(w)$ is true for all $w$, and suppose $\dep(w)<\ell_S(w)$.  If the depth of $w$ is not realized by a reduced factorization, we are done.  If the depth of $w$ is realized by a reduced factorization, then there exist $t_1,\ldots,t_r\in T$ with $w=t_1\cdots t_r$ where $\ell_S(t_i)>1$ for some $i$.  The reflection $t_i$ has a palindromic reduced expression $t_i=s_1\cdots s_{j-1}s_js_{j-1}\cdots s_1$ for some simple generators $s_1,\ldots, s_j \in S$, where $j=\dep(t_i)$~\cite[Ex. 1.10]{BBbook}.  

Since $w=t_1\cdots t_r$ is a reduced factorization, given any reduced expression for $t_i$, there exists a reduced expression for $w$ containing as a consecutive subexpression this reduced expression of $t_i$.  Therefore, $w$ has a reduced expression of the form $w=\cdots s_1\cdots s_{j-1}s_js_{j-1}\cdots s_1\cdots$.  In particular, this reduced expression has $s_{j-1}s_js_{j-1}$ consecutively, so $w$ is not short-braid-avoiding.
\end{proof}

Since depth is universally realized by reduced factorizations in the classical Coxeter groups, we have the following corollary.

\begin{cor}\label{cor:dep-length}
Let $w$ be an element of $S_n$, $B_n$, or $D_n$.  Then $\dep(w)=\ell_S(w)$ if and only if $w$ is short-braid-avoiding.
\end{cor}

Petersen and Tenner originally asked if the elements $w\in B_n$ with $\dep(w)=\ell_S(w)$ are precisely the fully commutative top-and-bottom elements defined by Stembridge~\cite{St2}.  We answer this in the affirmative by showing that the short-braid-avoiding elements are precisely the fully commutative top-and-bottom elements.

\begin{prop}\label{f.c is sba in B}
Let $W=B_n$.  Then $w\in W$ is a fully commutative top-and-bottom element if and only if it is short-braid-avoiding.
\end{prop}

\begin{proof}
Suppose $w\in W$ is short-braid-avoiding.  Then $w$ is $A$-reduced in the sense of Stembridge~\cite{St2}, since its reduced word cannot contain $s_0^Bs_1s_0^Bs_1$ or $s_1s_0^Bs_1s_2s_1s_0^Bs_1$.  Let $\tilde{w}\in A_n=S_{n+1}$ be the $A$-reduction of $w$.

Since $w$ is short-braid-avoiding, $\tilde{w}$ is fully commutative.  Hence all reduced expressions for $\tilde{w}$ are in the same commutativity class.  All commutativity relations in $A_n$ are also relations in $B_n$, so all reduced expressions for $\tilde{w}$ lift to reduced expressions for $w$.  Therefore, $w$ is the only element whose $A$-reduction is $\tilde{w}$.  Hence, it is a top and bottom element.

Now suppose $w\in W$ is not short-braid-avoiding.  If $w$ is not fully commutative, we are done.  If $w$ is fully commutative but not short-braid-avoiding, it must have $s_1s_0^Bs_1$ or $s_0^Bs_1s_0^B$ as consecutive generators in a reduced expression.  In the first case, $w$ is not a top element, and in the second, it is not a bottom element.
\end{proof}

\begin{rem} An immediate proof of this result can be obtained by using the characterization of fully commutative elements of type B given in \cite[\S 4.4]{BJN} and \cite[Corollaries 5.6 and 5.7]{St2}. In fact heaps of top-and-bottom fully commutative elements of type B are alternating heaps having zero or one occurrences of a label $s_0^B$.  In particular, by \cite[Theorem 1.14]{BJN}, they are in bijection with certain bicolored Motzkin paths \cite[Definition 1.11]{BJN}. This identification can be used to prove Corollary~\ref{cor:dep-length} in a manner analogous to Petersen and Tenner's proof for $S_n$.  (See \cite[\S 4.1]{PT}.)
\end{rem}

\begin{rem}
As was stated in Proposition \ref{f.c is sba in B}, in type $B$, an element is short- braid-avoiding if and only if it is fully commutative top-and-bottom.  Therefore, the first part of Corollary \ref{num dp=l} follows from Corollary \ref{cor:dep-length} and Proposition 5.9 (c) of \cite{St2}. The second part, regarding type $D$, is a consequence of Section 3.3 in \cite{St1} since, in $D_n$, the short-braid-avoiding elements are precisely the fully commutative ones. 
\end{rem}

Our characterization of when $\dep(w)=\ell_S(w)$ can also be stated using pattern avoidance by using Corollaries 5.6, 5.7, and 10.1 in~\cite{St2}.

\begin{thm}
Let $w\in B_n$.  Then $\dep(w)=\ell_S(w)$ if and only if $w$ avoids the following list of patterns: $$[\bar{1},\bar{2}],[\bar{2},\bar{1}],[1,\bar{2}], [3,2,1],[3,2,\bar{1}],[3,1,\bar{2}].$$
\end{thm}

\begin{thm}
Let $w\in D_n$.  Then $\dep(w)=\ell_S(w)$ if and only if $w$ avoids the following list of patterns:
\begin{multline*}
[\bar1,\bar2,\bar3], [1,\bar2,\bar3], [\bar2, \bar1, \bar3], [2,\bar1,\bar3], [\bar2, 1,\bar3], [2, 1, \bar3], 
[\bar3,\bar1,\bar2], [3,\bar1,\bar2], \\
[\bar3, 1, \bar2], [3,1,\bar2], [\bar3, 2, 1], [3, 2, 1], [\bar3, 2, \bar1], [3, 2, \bar1],
[\bar1, \bar3, \bar2], [1,\bar3,\bar2], \\
[\bar2, \bar3, \bar1], [\bar2, \bar3, 1], [2,\bar3,\bar1],[2,\bar3,1].
\end{multline*}
\end{thm}

\subsection{Coincidence of depth, length and reflection length}
\label{sec:four}

In this section we characterize the elements in a Coxeter group that satisfy $\ell_{T}(w)=\dep(w)$.
By \cite[Observation 2.3]{PT}, this is equivalent to having $\ell_{T}(w)=\ell_S(w)$. Actually, this characterization easily follows from results in \cite{Dyer}, \cite{PT}, and \cite{Tenner1}, all of which predate our work.  Nevertheless, we present it here for the sake of completeness.
\smallskip

Let $W$ be a Coxeter group.  Following Tenner~\cite{Tenner1}, we say an element $w \in W$ is {\bf boolean}
if the principal order ideal of $w$ in $W$, $B(w):=\{x \in W \mid x\leq w\}$ is a boolean poset, where $\leq$ refers to the strong Bruhat order. Recall that a poset is {\bf boolean} if it is isomorphic to the poset of subsets of $[k]$, ordered by inclusion, for some $k$.

Theorem 7.3 of \cite{Tenner1} states that an element $w \in W$ is boolean if and only if some (and hence any) reduced decomposition of $w$ has no repeated letters.  Furthermore, the following result is due to Dyer~\cite[Theorem 1.1]{Dyer}.

\begin{prop}\label{dyer}
Let $(W,S)$ be a Coxeter system, and let $w=s_1 \cdots s_n$ be a reduced decomposition of $w \in W$.
Then $\ell_{T}(w)$ is the minimum natural number $k$ for which there exist $1 \leq i_1
< \cdots <i_k \leq n$ such that $e=s_1 \cdots \hat{s}_{i_1} \cdots \hat{s}_{i_2} \cdots \hat{s}_{i_k} \cdots s_n$, where $\hat{s}$ indicates the omission of $s$.
\end{prop}

From these two results one can easily conclude that, for each $w \in W$, we have that $\ell_T(w)=\ell_{S}(w)$ if and only if $w$ is boolean. Hence by \cite[Theorem 7.4]{Tenner1} we get the following results.

\begin{thm}
Let $w \in B_n$. Then $\ell_{T}(w)=\dep(w)=\ell_S(w)$ if and only if $w$ avoids the following list of patterns: $$[\bar{1},\bar{2}],[\bar{2},\bar{1}],[1,\bar{2}], [3,2,1],[3,2,\bar{1}],[\bar{3},2,1], [3,\bar{2},1],[3,4,1,2],[3,4,\bar{1},2],[\bar{3},4,1,2].$$
\end{thm}

\begin{thm}
Let $w \in D_n$. Then $\ell_{T}(w)=\dep(w)=\ell_S(w)$ if and only if $w$ avoids the following list of patterns: 
\begin{multline*}
[\bar{1},\bar{2},\bar{3}] ,[\bar{1},\bar{3},\bar{2}],[\bar{2},\bar{1},\bar{3}],[\bar{2},\bar{3},\bar{1}],[\bar{3},\bar{1},\bar{2}],[\bar{3},\bar{2},\bar{1}] , [3,2,1], [3,4,1,2], \\
[3,2,\bar{1}],[3,\bar{1},\bar{2}], [3,4,\bar{1},\bar{2}],[3,4,\bar{2},\bar{1}],  [\bar{3},2,1], [\bar{2},\bar{3},1], [\bar{3},4,1,2],\\
[\bar{4},\bar{3},1,2], [1,\bar{2}] , [3,\bar{2},1], [\bar{3},2,\bar{1}], [\bar{3},4,\bar{1},2].
\end{multline*}
\end{thm}

Moreover by \cite[Corollaries 7.5 and 7.6]{Tenner1} we get a proof of Corollaries~\ref{thm:enumeration} and \ref{thm:enumerationD}.

\section{Open questions and further remarks}\label{sec:open}

There remain many possible further directions for the further study of depth.  First, we have analogues of the questions asked in~\cite[Section 5]{PT} for the symmetric group.  While we have enumerated the elements of maximal depth in $B_n$ and $D_n$, the number of elements of other, non-maximal depths remains unknown.

\begin{que}
How many elements of $B_n$ or $D_n$ have depth $k$?
\end{que}

For the symmetric group $S_n$, Guay-Paquet and Petersen found a continued fraction formula for the generating function for depth~\cite{GP-P}.

Petersen and Tenner also asked the following question, which we now extend to $B_n$ and $D_n$:

\begin{que}
\label{que:shallow}
Which elements of $B_n$ or $D_n$ have $\dep(w)=(\ell_T(w)+\ell_S(w))/2$?
\end{que}

Furthermore, it seems possible that variations of our techniques can be extended to the infinite families of affine Coxeter groups, for which combinatorial models as groups of permutations on $\mathbb{Z}$ are given in~\cite[Chapter 8]{BBbook}.

\begin{que}
What are the analogues of Theorems~\ref{thm:depthB} and~\ref{thm:depthD} for the infinite families of affine Coxeter groups?
\end{que}

Given Examples~\ref{Brefllengthcounterex} and~\ref{Drefllengthcounterex}, we can ask the following:

\begin{que}
For which elements of $B_n$ and $D_n$ can depth be realized by a product of $\ell_T(w)$ reflections?
\end{que}

We also ask some questions relating to Theorems~\ref{reduced} and~\ref{Dreduced}.

\begin{que}
Is depth universally realized by reduced factorizations for all Coxeter groups?  If so, is there a uniform proof?  If not, can one characterize the elements of Coxeter groups whose depth is realized by a reduced factorization? 
\end{que}

It would be interesting to know the answer even for various specific Coxeter groups.  For example, one might answer this question for the infinite families of affine Coxeter groups.  The question is interesting even for the finite exceptional Coxeter groups.  One can attempt to use a computer to find the answer in this case, but finding a feasible method for computing the answer seems nontrivial for $E_8$ or even $E_7$.

Furthermore, there is another perspective on Theorems~\ref{reduced} and~\ref{Dreduced} that leads to further questions.

Given a Coxeter group $(W,S)$ and an element $w\in W$, define the {\bf reduced reflection length} $\ell_R(w)$ by
\begin{equation}\label{reflength}
\ell_R(w):=\mbox{min} \{r \in {\mathbb N} \mid w=t_{1}\cdots t_{r}\; 
\mbox{for} \;  t_{1},\ldots,t_{r} \in T \ {\rm and} \ \ell_S(w)=\sumlim_{i=1}^r \ell_S(t_i) \}.
\end{equation}

Note that, by definition $\ell_T(w)\leq \ell_R(w)\leq \ell_S(w)$.  
\smallskip

For example, $w=[4,2,5,1,3] \in S_5$ has a reduced expression $w=s_3s_4s_1s_2s_1s_3=s_3s_4(s_1s_2s_1)s_3$. Hence $\ell_S(w)=6$, and one can check that $\ell_R(w)=4$.  However, its reflection length is equal to 2 since $w=(s_3s_4s_3)(s_3s_1s_2s_1s_3)=t_{35}t_{14}$. Hence in this case  $\ell_T(w) < \ell_R(w) < \ell_S(w)$.  
\smallskip

Reduced reflection length is related to depth as follows.

\begin{prop}
Let $(W,S)$ be a Coxeter group and $w\in W$.  Then
$$\dep(w)\leq \frac{\ell_R(w)+\ell_S(w)}{2}.$$
If the depth of $w$ is realized by a reduced factorization, then we have equality.
In particular, for $w$ in a classical finite Coxeter group, $\dep(w)=(\ell_R(w)+\ell_S(w))/2$.
\end{prop}

\begin{proof}
There must exist some $t_1,\ldots,t_r$ so that $w=t_1\cdots t_r$ realizes $\ell_R(w)$, meaning that $\ell_R(w)=r$ and $\ell_S(w)=\sum_{i=1}^r \ell_S(t_i)$.  Hence
$$\dep(w)\leq\sumlim_{i=1}^r\dep(t_i)
=\sumlim_{i=1}^r \frac{1+\ell_S(t_i)}{2}
=\frac{\ell_R(w)+\ell_S(w)}{2}.$$

Now suppose the depth of $w$ is realized by a reduced factorization.  Then there exist $t_1,\ldots,t_r$ with
$w=t_1\cdots t_r$ and $\ell_S(w)=\sum_{i=1}^r \ell_S(t_i)$.  Therefore, $\ell_R(w)\leq r$.  Moreover,
$$\dep(w)=\sumlim_{i=1}^r \frac{1+\ell_S(t_i)}{2}=\frac{r+\ell_S(w)}{2}.$$
Since $\ell_R(w)\leq r$, we have
$$\dep(w)\geq \frac{\ell_R(w)+\ell_S(w)}{2}.$$
\end{proof}

One can give an alternate definition of reduced reflection length as follows.  We have $\ell_S(wt)=\ell_S(w)+\ell_S(t)$ if and only if $w <_R wt$ in right weak order.  Hence, $\ell_R(w)$ is the length of the shortest chain $e=w_0 <_R \cdots <_R w_r=w$ in right weak order where, for all $i\in[r]$, $w_i=w_{i-1} t$ for some reflection $t$.  Given a partial order $\prec$ on $W$, define $\ell_{\prec}(w)$ to be the length of the shortest chain $e=w_0\prec\cdots\prec w_r=w$ where, for all $i\in[r]$, $w_i=w_{i-1} t$ for some reflection $t$.  If $\prec$ is Bruhat order, then $\ell_\prec=\ell_T$, and if $\prec$ is right (or left) weak order, $\ell_\prec=\ell_R$, a formula for which (for $S_n$, $B_n$, and $D_n$) is given above.

Hence, for any partial order on $W$ (or at least partial orders whose relations are a subset of the relations of Bruhat order), we can ask the following.

\begin{que}
Find formulas for $\ell_\prec$ for other partial orders on Coxeter groups.  Determine for which elements $\ell_\prec=\ell_S$.
\end{que}

We also have the following generalization of Question~\ref{que:shallow}.

\begin{que}
Which elements of $W$ have $\ell_\prec(w)=\ell_T(w)$?
\end{que}

A particularly interesting family of partial orders are the sorting orders of Armstrong~\cite{Arm}, which were further studied by Armstrong and Hersh~\cite{AH}.  These partial orders contain all the relations of weak order but are contained in Bruhat order.

\def\cprime{$'$}


\begin{thebibliography}{1}

\bibitem{Arm}
D.~Armstrong.
\newblock The sorting order on a Coxeter group.
\newblock {\em J. Combin. Theory Ser. A}, 116(8):1285--1305, 2009.

\bibitem{AH}
D.~Armstrong and P.~Hersh.
\newblock Sorting orders, subword complexes, Bruhat order and total positivity.
\newblock {\em Adv. in Appl. Math.}, 46(1--4):46--53, 2011.

\bibitem{BJN}
R.~Biagioli, F.~Jouhet, and P.~Nadeau.
\newblock Fully commutative elements in finite and affine {C}oxeter groups.
\newblock {\em \emph{To appear in} Monatsh. Math., {\em DOI} 10.1007/s00605-014-0674-7, \emph{preprint}
  arXiv:1402.2166}, 2014.

\bibitem{BJS}
S.~C. Billey, W.~Jockusch, and R.~P. Stanley.
\newblock Some combinatorial properties of {S}chubert polynomials.
\newblock {\em J. Algebraic Combin.}, 2(4):345--374, 1993.

\bibitem{BBbook}
A.~Bj{\"o}rner and F.~Brenti.
\newblock {\em Combinatorics of {C}oxeter groups}, volume 231 of {\em Graduate Texts in Mathematics}.
\newblock Springer, New York, 2005.

\bibitem{DiGr}
P. Diaconis and R.~L. Graham.
\newblock Spearman's footrule as a measure of disarray.
\newblock {\em J. Roy. Statist. Soc. Ser. B}, 39(2):262--268, 1977.

\bibitem{Dyer}
M.~J. Dyer.
\newblock On minimal lengths of expressions of {C}oxeter group elements as products of reflections.
\newblock {\em Proc. Amer. Math. Soc.}, 129(9):2591--2595, 2001.

\bibitem{Fan}
C.~K.~Fan.
\newblock Schubert varieties and short braidedness.
\newblock {\em Transformation Groups}, 3(1):51--56, 1998.

	
\bibitem{GP-P}
M. Guay-Paquet and T.~K. Petersen.
\newblock The Generating Function for Total Displacement.
\newblock {\em Electon. J. Combin.}, 21(3) : 2014,  Paper P3.37

\bibitem{Kn}
Donald~E. Knuth.
\newblock {\em The art of computer programming. {V}ol. 3}.
\newblock Addison-Wesley, Reading, MA, 1998.
\newblock Sorting and searching, Second edition.

\bibitem{PT}
T.~K. Petersen and B.~E. Tenner.
\newblock The depth of a permutation.
\newblock {\em J. Comb.}, 6(1--2):145--178, 2015.

\bibitem{St1}
J.~R. Stembridge.
\newblock The enumeration of fully commutative elements of Coxeter groups
\newblock {\em J. Algebraic Combin.}, 7: 291--320, 1998.

\bibitem{St2}
J.~R. Stembridge.
\newblock Some combinatorial aspects of reduced words in finite {C}oxeter groups.
\newblock {\em Trans. Amer. Math. Soc.}, 349(4):1285--1332, 1997.

\bibitem{Tenner1}
B.~E. Tenner.
\newblock Pattern avoidance and the {B}ruhat order.
\newblock {\em J. Combin. Theory Ser. A}, 114(5):888--905, 2007.

\end{thebibliography}
\end{document}